\documentclass[11pt]{amsart}
\usepackage[letterpaper,top=1in, left=1in, right=1in, bottom=1in]{geometry}
\usepackage{amsfonts, amstext, amsmath, amsthm, amscd, amssymb} 
\usepackage{mathtools}
\usepackage{float} 
\usepackage{graphics}
\usepackage{caption}
\usepackage{subcaption}
\usepackage{import} 
\usepackage[abs]{overpic}

\usepackage{physics}  
 
\usepackage{microtype}
\usepackage[hidelinks,pagebackref,pdftex]{hyperref}
\usepackage[dvipsnames]{xcolor}
\definecolor{MyCyan}{HTML}{00F9DE}
\usepackage{marginnote}
\long\def\@savemarbox#1#2{\global\setbox#1\vtop{\hsize\marginparwidth 
  \@parboxrestore\tiny\raggedright #2}}
\newcommand{\RR}{\mathbb{R}}  
  
\newcommand{\ZZ}{\mathbb{Z}}

\renewcommand{\SS}{\mathbb{S}}

\renewcommand{\setminus}{{\smallsetminus}}

\newtheorem{theorem}{Theorem}[section]
\newtheorem*{theorem*}{Theorem}
\newtheorem{proposition}[theorem]{Proposition}
\newtheorem{lemma}[theorem]{Lemma}
\newtheorem{corollary}[theorem]{Corollary}

\newtheorem*{namedtheorem}{\theoremname}
\newcommand{\theoremname}{testing}
\newenvironment{named}[1]{\renewcommand{\theoremname}{#1}\begin{namedtheorem}}{\end{namedtheorem}}

\theoremstyle{definition}
\newtheorem{definition}[theorem]{Definition}

\newtheorem{conjecture}[theorem]{Conjecture}
\newtheorem{remark}[theorem]{Remark}

\newtheorem{algorithm}[theorem]{Algorithm}

\newcommand{\refthm}[1]{Theorem~\ref{Thm:#1}}
\newcommand{\reflem}[1]{Lemma~\ref{Lem:#1}}
\newcommand{\refprop}[1]{Proposition~\ref{Prop:#1}}
\newcommand{\refcor}[1]{Corollary~\ref{Cor:#1}}

\newcommand{\refdef}[1]{Definition~\ref{Def:#1}}
\newcommand{\refsec}[1]{Section~\ref{Sec:#1}}
\newcommand{\reffig}[1]{Figure~\ref{Fig:#1}}

\title{A robot that unknots knots} 

\author[Hui]{Connie On Yu Hui} 
\address[]{School of Mathematics, Monash University, Australia } 
\email[]{onyu.hui@monash.edu | connieonyuhui@gmail.com} 

\author[Ibarra]{Dionne Ibarra}
\address{School of Mathematics, Monash University, Australia}
\email[]{dionne.ibarra@monash.edu} 

\author[Kauffman]{Louis H. Kauffman} 
\address[]{Math Dept, University of Illinois at Chicago\newline \indent
851 South Morgan Street, Chicago, IL 60607-7045, USA, and \newline \indent
International Institute for Sustainability with Knotted Chiral Meta Matter (WPI-SKCM2),\newline \indent
Hiroshima University, 1-3-1 Kagamiyama, Higashi-Hiroshima, Hiroshima 739-8526, Japan
} 
\email[]{loukau@gmail.com} 

\author[McQuire]{Emma N. McQuire} 
\address[]{School of Mathematics, Monash University, Australia } 
\email[]{Emma.McQuire@monash.edu} 

\author[Montoya-V.]{Gabriel Montoya-Vega} 
\address[]{Department of Mathematics, University of Puerto Rico at R\'io Piedras, San Juan, PR, USA} 
\email[]{gabrielmontoyavega@gmail.com | gabriel.montoya@upr.edu} 

\author[Mukherjee]{Sujoy Mukherjee} 
\address[]{Department of Mathematics, University of Denver, USA} 
\email[]{sujoymukherjee.math@gmail.com | sujoy.mukherjee@du.edu} 

\author[Reid]{Corbin Reid} 
\address[]{School of Mathematics, Monash University, Australia } 
\email[]{Corbin.Reid1@monash.edu}



\begin{document} 

\begin{abstract} 
Consider a robot that remembers only the starting position and walks along a knot once on a knot diagram, switching every undercrossing it meets until it returns to the starting position. 
We observe that the robot produces an ascending diagram, and we provide a new combinatorial proof that every ascending or descending knot diagram can be transformed into the zero-crossing unknot diagram.  Using the machinery developed from the combinatorial proof, we show that the minimal number of Reidemeister moves required for such a transformation is bounded above by $(7C+1)C$ if the diagram has $C$ crossings. 
Moreover, we provide a new alternative proof that there exist sequences of Reidemeister moves that do not increase the number of crossings and transform ascending or descending knot diagrams into zero-crossing unknot diagrams.  
\end{abstract}

\maketitle
\tableofcontents
\section{Introduction} 
Given a knot diagram of a knot, a standard way to unknot the knot is to switch some crossings in the diagram. The unknotting method presented in \cite[p.58-59]{Adams:TheKnotBook} (see also \cite{AlexanderConwayPoly, PrzytyckiTraczyk}) traverses the knot in the diagram from a given starting point in a given direction once. If we arrive at a certain crossing for the first time along the understrand, we switch it. Otherwise, we do not do any switching. Such a method produces a descending diagram, which is an unknot diagram. The unknottedness of a descending or ascending knot diagram is immediately obvious if we pull up the starting point, see Figure~\ref{Fig:3dto2d} for an example.

\begin{figure}[ht]
    \centering
     $$ \vcenter{\hbox{\begin{overpic}[scale=1]{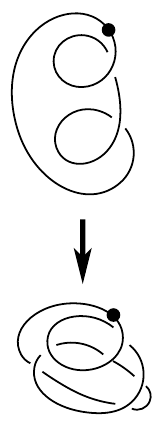}
     \put(100, 150){$\hookrightarrow \mathbb{R}^3$}
      \put(100, 20){$\hookrightarrow \mathbb{R}^2$}
\end{overpic}}}$$ 
    \caption{Illustration of a projection of an ascending diagram viewed from above, with starting point marked. If one pulls up the starting point, the diagram unknots.} 
    \label{Fig:3dto2d}
\end{figure}

Observe that the aforementioned standard unknotting algorithm requires us to remember which crossings have been met as we trace along the knot in the diagram. In this paper, we consider a relatively memoryless algorithm (which only requires the robot to remember the starting position): 

\begin{algorithm}[Unknotting robot] \label{Alg:Robot}
    Given a planar knot diagram of a knot $K$, choose and fix a starting point (that is not a crossing) on the knot and an orientation. Place a \textit{robot} on the knot at the starting point. The robot will follow the orientation, check at every point in the knot (except at crossings) whether it is at the starting point\footnote{In practice, this can be avoided by turning the knot diagram into a long knot diagram with a starting and ending point in which case the algorithm ends when the robot reaches the ending point.}, and stop when it comes back to the starting point. When it reaches an overcrossing, it passes through, but the robot is scared of the dark; when it reaches an undercrossing, it switches it to an overcrossing before passing through. See Figure~\ref{Fig:ActionofRobot}. 
\end{algorithm} 

\begin{figure}[ht]
    \centering
\includegraphics{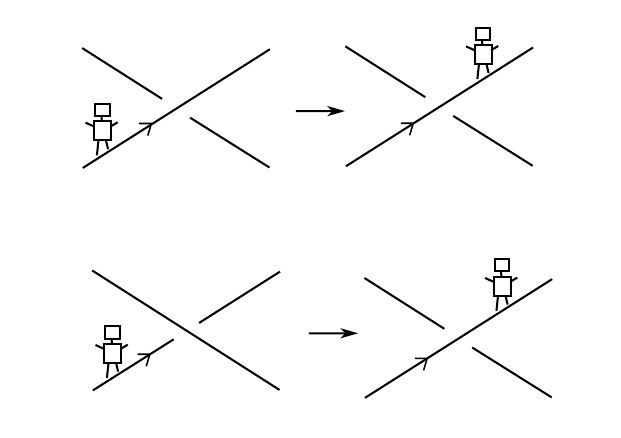}
    \caption{Action of robot at overcrossing (top) and undercrossing (bottom). Arrow shows direction of robot's travel. }
    \label{Fig:ActionofRobot}
\end{figure} 

Observe that the robot meets every crossing exactly twice. If it meets a crossing along an overstrand first, it will switch it the second time it walks past the crossing; if it meets a crossing along an understrand first, it will switch the crossing twice, leaving it unchanged. As we must meet each crossing along understrands first as we trace along an ascending diagram, the above proves the following statement: 

\begin{theorem}\label{Thm:UnknottingKnots}
The robot in Algorithm~\ref{Alg:Robot} produces an ascending diagram. 
\end{theorem}


The paper \cite{PrzytyckiTraczyk} by Przytycki and Traczyk provides a combinatorial proof for the unknottedness of ascending and descending diagrams and is probably the first paper to do so via the Reidemeister moves. A proof is also given by Bazier-Matte and Schiffler in \cite{BazierMatte-Schiffler:ClusterAlgebrasII}. A main goal of our paper is to provide a new combinatorial proof for the fact that an ascending or descending diagram is an unknot diagram. Our combinatorial proof has several advantages. First, it gave us insight and tools for proving an upper bound of the minimal number of Reidemeister moves required for transforming an ascending or descending knot diagram into the zero-crossing unknot diagram (see Theorem~\ref{Thm:UpBdRMoves}). Second, the machinery from the combinatorial proof also allows us to provide a new alternative proof that an ascending or descending knot diagram can be transformed into an unknot diagram using Reidemeister moves that do not increase the number of crossings. Such a combinatorial proof can be understood by Flatlanders living in the Flatland of $\mathbb{R}^2$, who do not have enough dimensions to understand geometric arguments such as the one shown in \reffig{3dto2d}. In this paper, we are going on an expedition down under to provide a combinatorial proof that can be understood by Flatlanders. 

\refsec{UnknottingKnots} details the combinatorial arguments, which can be outlined as follows:
First, we describe the action of the robot in terms of the knot diagrams' Gauss codes and show that the robot action produces ascending knot diagrams. Second, we introduce loop detour moves, explain how these moves affect the Gauss code of a knot diagram, and use these moves to show that an ascending knot diagram can be transformed into an unknot diagram.  



\subsection{An upper bound on the minimal number of Reidemeister moves for ascending diagrams}
Using the machinery developed from the combinatorial proof, we provide an upper bound for the minimal number of Reidemeister moves required for transforming an ascending or descending diagram into an unknot diagram. 

In 2001, Hass and Lagarias published a paper \cite{Hass-Lagarias:NumRMovesForUnknotting} showing that there exists a positive constant $K$ such that for each positive integer $c$, any unknotted knot diagram with $c$ crossings can be transformed to the zero-crossing knot diagram using at most $2^{Kc}$ Reidemeister moves. Their bound is exponential with respect to the number of crossings. In 2015, Lackenby gave an improved upper bound in \cite{Lackenby:PolynomialUpBdRMoves}. The paper showed that any diagram of the unknot with $c$ crossings may be transformed into the zero-crossing unknot diagram using at most $(236c)^{11}$ Reidemeister moves. Instead of an exponential growth with respect to the number of crossings $c$, Lackenby gave an upper bound that is a polynomial function of $c$. 

 In this article, we show an upper bound that is quadratic with respect to the number of crossings for ascending or descending knot diagrams: 

\begin{named}{\refthm{UpBdRMoves}}
    Any ascending or descending knot diagram with $C$ crossings can be transformed into the zero-crossing unknot diagram using at most $(7C+1)C$ Reidemeister moves. 
\end{named} 

A key concept we used for showing \refthm{UpBdRMoves} is the loop detour move, which, roughly speaking, is a move on a one-crossing kink encircling a tangle diagram (see \reffig{LoopTangleDiag} for example) such that the loop stays away from the tangle after the move. Note that the sequence of Reidemeister moves we construct for each loop detour move may contain moves that increase the number of crossings in the diagram.

Other related works include the following: C. Hayashi, M. Hayashi, Sawada, and Yamada in \cite{Hayashi-Hayashi-Sawada-Yamada_MinUnknottingSeq} showed that each member in a special collection of ascending diagrams can be transformed into the trivial diagram by a minimal sequence of Reidemeister moves. In \cite{Hass-Nowik:UnknotRequireQuadraticNum}, Hass and Nowik gave a lower bound for the number of Reidemeister moves required to turn each member in a family of unknotted knot diagrams to a zero-crossing knot diagram.  Their lower bound is a quadratic polynomial function of the number of crossings.

\subsection{Simplification of ascending diagrams}
Unknotting problems also consist of questions regarding the \emph{monotonic simplification} of diagrams of unknots; these are sequences of Reidemeister moves that do not increase the number of crossings in the diagram at any stage. Due to Goeritz in \cite{Goeritzhardunknots} and Kauffman and Lambropoulou in \cite{KLHardunknots}, there are well known examples of \textit{hard unknots}: diagrams of unknots that cannot be monotonically simplified via a sequence of Reidemeister moves. For succinctness, we will call a monotonic simplification a \textit{simplification} and if there exists a monotonic simplification of a knot diagram, we will say that the knot diagram can be \textit{simplified}. 

We used the machinery from the combinatorial argument to prove that any ascending knot diagram can be simplified. After writing the proof, we found Section~6 of a recent paper by Bazier-Matte and Schiffler~\cite{BazierMatte-Schiffler:ClusterAlgebrasII}, which proves a simplification result for ascending diagrams. We were also directed to Lemma~2.14 in a paper~\cite{PrzytyckiTraczyk} written by Przytycki and Traczyk (published in 1988), which gives a combinatorial proof that ascending diagrams can be simplified to an unknot diagram by Reidemeister moves \cite{PrzytyckiTraczyk}.  In the present paper, we provide a new alternative proof that ascending diagrams can be simplified by Reidemeister moves that use our loop-tangle machinery. 

\begin{named}{\refcor{Simplified}}
An ascending diagram can always be simplified by Reidemeister moves to a crossingless unknot diagram. 
\end{named}  

Since this result is fundamental to skein theory, here related to the robot unknotting process, we feel that the exposition of yet another proof of the result is justified.
Part of the originality of our paper arises in the relationship between ascending diagrams and the robot unknotter. 

This result is a corollary to Theorem~\ref{thm:always_simplify} which states that it is always possible to simplify the interior of a loop-tangle subdiagram in an ascending diagram such that the associated loop detour move may be performed as a simplification. Observe that there exist loop-tangle subdiagrams that cannot be immediately simplified via Reidemeister~I or II moves, see Figure~\ref{fig:complicated_loop-tangle} for example. 



\subsection{Interdisciplinary connections}
Unknotting problems can also be seen across disciplines. In \cite{Goldman-Kauffman:KnotsTanglesAElectricalNetworks}, Goldman and Kauffman provided a translation of Reidemeister moves on link diagrams to graphical moves on signed graphs. Their paper reviewed the background on classical electrical networks which can be modelled by a graph with edges assigned with resistance. The study of such electrical networks motivates a connection between knot invariants and conductance-preserving transformations on networks. Using the knot-theoretic machinery we develop in this paper, we can prove the following graph-theoretic statement.

\begin{named}{\refcor{PlanarGraph}}
    Any planar graph can be transformed into a finite collection of points via a sequence of the graphical moves.
\end{named} 

In \cite{Wang}, L. C. Wang studied DNA topoisomerases and observed that these enzymes have unknotting properties on DNA circles. DNA topoisomerases are enzymes that create double strand break and rejoin in DNA circles as discussed in \cite{WANGLUIDNA}. We use blackboard framed projections of framed knots to extend the unknotting robot to framed knots and compare it to a hypothesized model of DNA topoisomerases.

\subsection{Acknowledgements}  
The project started during the MATRIX Research program: Low Dimensional Topology: Invariants of Links, Homology Theories, and Complexity. We thank the mathematical research institute MATRIX in Australia where part of this research was performed.
DI acknowledges the support by the Australian Research Council grant DP210103136 and DP240102350. It gives LHK pleasure to recall that the Robot made its first appearance in his conversations with Allison Henrich and that the Robot will make further appearances in their work on knot magic. 
GMV was supported by the Matrix-Simons travel grant and acknowledges the support of the National Science Foundation through Grant DMS-2212736. SM was supported by the Matrix-Simons travel grant and is grateful to P. Vojt\v{e}chovsk\'{y} for his support through the Simons Foundation Mathematics and Physical Sciences Collaboration Grant for Mathematicians no. 855097. We thank Józef Przytycki for informing us of Lemma~2.14 in~\cite{PrzytyckiTraczyk}.

\section{Preliminaries}\label{Sec:Prelim} 

We seek to show that the robot operation always gives us an ascending diagram and to provide a combinatorial argument for the fact that an ascending diagram is an unknot diagram.  This section provides preliminaries on how we use Gauss codes to combinatorially prove the aforementioned statements in later sections. 

\begin{definition}\label{Def:GaussCode} \ 
\begin{enumerate}
    \item  Given an oriented knot diagram $D$ with $n$ crossings and a starting point $b$ on $D$, we label each crossing with a unique symbol and call each such symbol a \emph{crossing label}. 
    \item The \emph{Gauss code} of the diagram $D$ based at $b$ is the sequence of $2n$ letters that is obtained by following the orientation of $D$ from $b$ and recording the crossings met in order and whether each crossing is traversed under or over. 
    \item Each \emph{letter} in a Gauss code is of the form $x^\epsilon$, where $x$ denotes the crossing label and $\epsilon\in\{-1,+1\}$. A crossing $x$ traversed by an overcrossing is marked $x^{+1}$ and a crossing traversed by an undercrossing is marked $x^{-1}$. 
    \item An \emph{empty Gauss code} is a Gauss code with no letters.  
\end{enumerate}
\end{definition}

Note that each crossing label always appears exactly twice in a Gauss code, once as an overcrossing and once as an undercrossing. 

An empty Gauss code corresponds to an unknot diagram with no crossings. 

\begin{definition}
    A diagram that is \emph{ascending from $b$}, or simply an \emph{ascending diagram}, is an oriented diagram for which, following the diagram from a starting point $b$, crossings are first met as under- and then as an overcrossing. A \emph{descending diagram} is the opposite; crossings are first met as over- then as undercrossings. Equivalently, each type of diagram is the reverse orientation of the other.
\end{definition} 

Note that, for a given diagram, each choice of arc between crossings and direction of travel gives a potentially distinct ascending diagram; in a diagram with $n$ crossings, this is up to $4n$ distinct ascending diagrams. Another way to see this is that there are four directions a crossing may be first met when ascending from starting points in the arcs around it, each of which gives a potentially distinct ascending diagram.

\begin{lemma} \label{Lem:GaussCodesForAscending}
Let $D$ be an oriented knot diagram. Let $\omega$ be the Gauss code obtained after fixing an arbitrary starting point $b$.  The diagram $D$ is ascending from $b$ if and only if the first and second appearances of each crossing label in $\omega$ are under and over, respectively. The diagram is descending from $b$ if the labels instead appear as over before under. 
\end{lemma} 

\begin{proof}
    Given a starting point $b$ and a walking direction, the diagram $D$ is ascending from $b$ if and only if every crossing is met under before over, if and only if each crossing appears first as an undercrossing then as an overcrossing in the Gauss code. 
    
    The argument for the case of a descending diagram follows similarly from the definition of a descending diagram and the definition of Gauss code.
\end{proof} 

One of the key ideas for showing the main results in \refsec{UnknottingKnots} is the use of loop detour moves. This section provides the details of detour moves, tangle detour moves, and loop detour moves. 

\begin{definition}
    Given two points $a_1, a_2$ on a knot or link diagram connected by an arc $\gamma$ that meets only undercrossings (or only overcrossings resp.), a \emph{detour move} is an ambient isotopy of $\gamma$ such that the following hold: 
    \begin{enumerate}
        \item In the ambient isotopy, $\gamma$ moves beneath (resp. above) the rest of the diagram, 
        \item a link diagram is obtained after the move, and
        \item the deformed $\gamma$ either does not meet any crossing or meets only undercrossings (resp. overcrossings). 
    \end{enumerate}
    We call $\gamma$ the \emph{detour arc}. See \reffig{DetourMove} for example. 
\end{definition} 

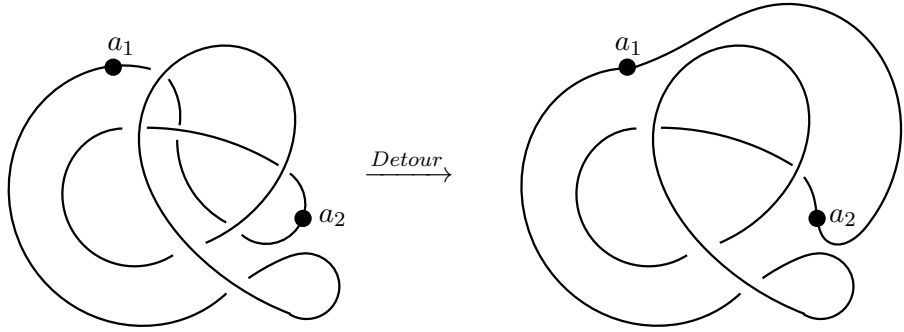
\begin{figure}[ht]
    \centering
     $$ \vcenter{\hbox{\begin{overpic}[scale=1]{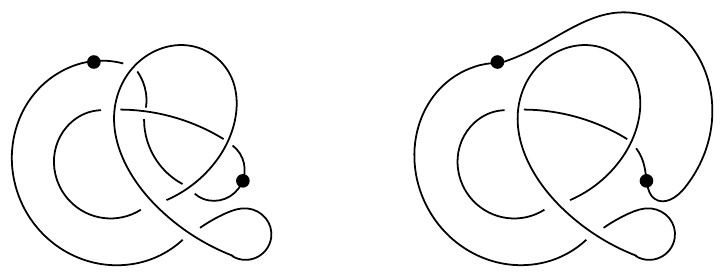}
\put(140, 60){$\xrightarrow[]{Detour}$}
\put(43, 110){$a_1$}
\put(123, 45){$a_2$}
\put(235, 110){$a_1$}
\put(316, 45){$a_2$}
\end{overpic}}}$$ 
    \caption{Illustration of a detour move.}
    \label{Fig:DetourMove}
\end{figure}

\begin{definition} \label{Def:TangDetourMove} 
    Let $D$ be a knot diagram of some knot $K$. 
    Let $T$ be a tangle sub-diagram of $D$ consisting of $\frac{m+n}{2}$ disjoint subarcs of $K$, where $m$ and $n$ are nonnegative integers whose sum is even. 

    An \emph{$(m,n)$-tangle detour move $M$ corresponding to $T$} is a detour move over the tangle $T$ such that the following hold: 
    \begin{enumerate}
        \item the detour arc $\gamma$ of $M$ is a subarc of $K$ disjoint from $T$, and
        \item $\gamma$ meets $T$ at exactly $m$ distinct crossings before the move and exactly $n$ distinct crossings after the move. 
    \end{enumerate}
    We call $T$ the \emph{$(m,n)$-tangle diagram} associated to the tangle detour move $M$. 
\end{definition} 

See \reffig{TangleDetourMove} for an example of a $(6,4)$-tangle detour move and observe that \reffig{DetourMove} shows a detour move that is not a tangle detour move. 

\begin{figure}[H]
    \centering
    \includegraphics[width=0.8\linewidth]{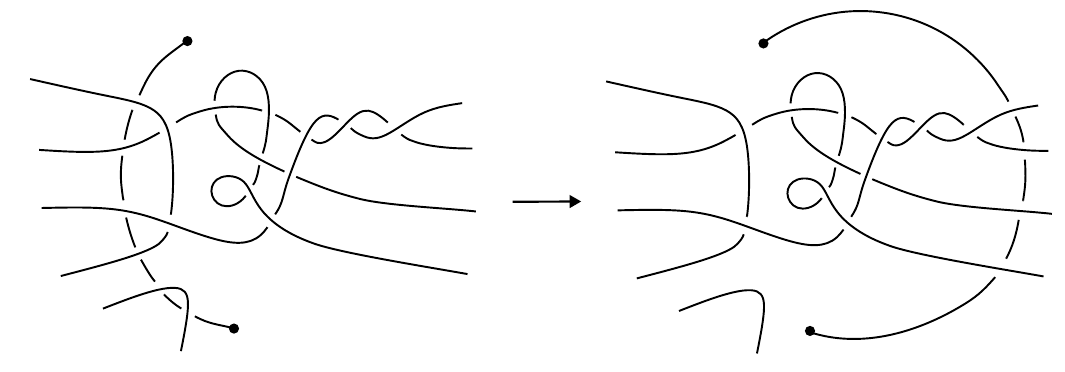}
    \caption{A (6,4)-tangle detour move. The detour arc is also shown.}
    \label{Fig:TangleDetourMove}
\end{figure}

Let $\mathcal{D}$ be the closed disc that $\gamma$ sweeps over. Following from the definition of an $(m,n)$-tangle detour move, the (possibly zero) integers $m$ and $n$ have the same parity. 

Note that an $(m,n)$-tangle $T$ in a knot diagram consists of $\frac{m+n}{2}$ arcs in the closed disc $\mathcal{D}$. The intersection between the boundary of $\mathcal{D}$ and the arcs are exactly the endpoints of all the arcs.  We call the sub-graph obtained by disregarding the over/under-crossing information in the sub-diagram corresponding to $T$ an \emph{$(m,n)$-tangle graph}, or simply a \emph{tangle graph} when the context is clear.

Given a tangle graph, we call each edge with exactly one vertex as an endpoint a \emph{sprout} and we call each edge with two (possibly the same) vertices as endpoints a \emph{non-sprout edge}. For any strand in the tangle graph that does not have any vertices, we call it a \emph{trivial arc}. 

\begin{lemma} \label{Lem:NoCommonEdge} 
    If $P$ is a polygon in a planar $4$-valent graph $G$, then for any two sides of $P$, they cannot be the same edge in $G$.
\end{lemma} 

\begin{proof}
    Suppose it were true that $P$ is a polygon in $G$ and there exist two sides of $P$ that are the same edge in $G$. Let $w$ be the number of sides of $P$, and assume the sides of $P$ are consecutively labelled with integers $1, 2, \ldots, w$ respectively. Suppose Side $i$ and Side $j$ of $P$ are the same edge $E$ in $G$ with $i<j$. Since $G$ is $4$-valent, we cannot have $1$-valent vertex, so $i\neq (j-1)$, and thus $i<(j-1)$.  
    
    We would have $a$-sided polygon $P_a$ (with $a\leq (j-i-1)$) contained in a $b$-sided polygon $P_b$ (with $b\leq (w-(j-i-1))$), or the other way around, and the graph edge $E$ connects $P_a$ and $P_b$. See \reffig{SameSidesijInP} for example. Without loss of generality, assume $P_a$ is the innermost such polygon (see \reffig{Innermost} for example).

    Note that $P_a$ cannot have more than one side being identified to any other side of $P$ because otherwise, $P$ would have a different number of sides. Since $G$ is a $4$-valent graph and $E$ is externally incident to a vertex $V$ of $P_a$, there would be exactly one other edge $E'$ that is internally incident to $V$. Observe that the number of internal incident edges to each of the other vertices of $P_a$ is two, so the total number of edges that are internally incident to all vertices of $P_a$ would be an odd number. However, as the graph is $4$-valent, these internal incident edges must be pairwise connected up, which implies an even number of internal incident edges. Contradiction. 
\end{proof}

\begin{figure}[ht]
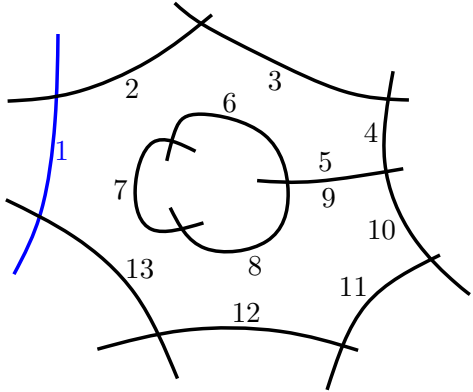  
\caption{An imaginary example that helps explain the proof of \reflem{NoCommonEdge}: A $4$-valent subgraph with a $13$-sided polygon $P$. Side 5 and Side 9 share the same edge $E$, which connects a $3$-sided polygon $P_3$ and an $8$-sided polygon $P_8$. A planar $4$-valent graph cannot have such a subgraph because the number of internally incident edges for $P_3$ can only be even, not odd.}
\label{Fig:SameSidesijInP}
\end{figure}  

\begin{figure}[ht]
    \centering
\includegraphics[width=0.5\linewidth]{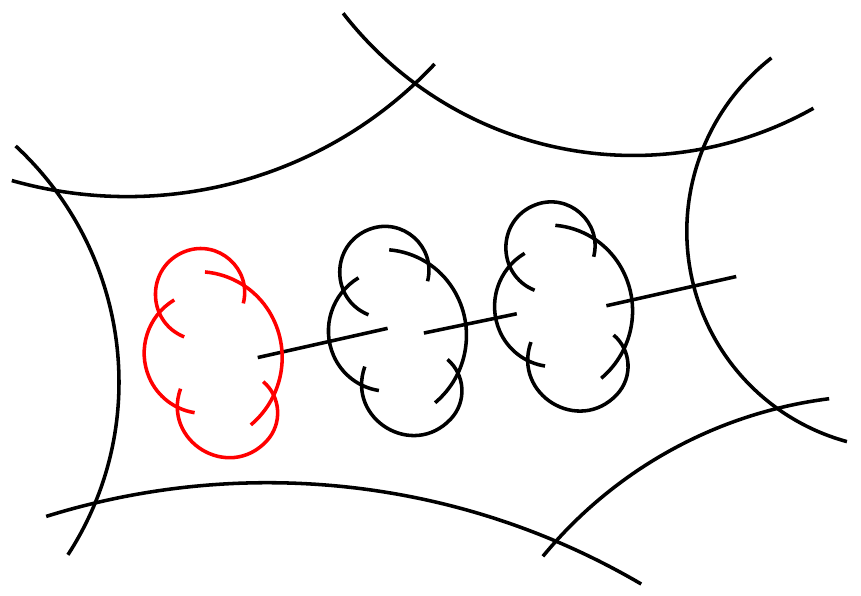}
    \caption{Another faux example highlighting in red the innermost polygon.}
    \label{Fig:Innermost}
\end{figure}

In the next lemma, we relate the number of edges and the number of vertices in a tangle graph.  

\begin{lemma} \label{Lem:RelatingVAndE}
    Let $G_T$ be an $(m,n)$-tangle graph. Let $v$, $e$, and $t$ be the number of vertices, number of non-sprout edges, and number of trivial arcs in $G_T$ respectively. We have 
    \[ 4v = 2e + m+n-2t.\]
\end{lemma}  

\begin{proof}
Any $(m,n)$-tangle graph has $(m+n-2t)$ sprouts. 
Note that each vertex in $G_T$ has four incident edges (including sprouts). Each non-sprout edge is touching exactly two (possibly the same) vertices. By counting the four edges around every vertex (some edges are possibly the same), we count each non-sprout edge in $G_T$ twice and each sprout once. Therefore, we have $4v = 2e + (m+n-2t)$. 
\end{proof} 

From here on, the graph theoretic notations or terminology may be used interchangeably with the corresponding diagrammatic ones. For example, we may use the term \emph{edge in a diagram $D$} to refer to the corresponding arc in the diagram that gives an edge in the associated graph obtained by disregarding the over/under-crossing information in $D$.

For the purpose of showing the results in later sections, we define another special case of detour moves:

\begin{definition} \label{Def:LoopDetour}
    Let $D$ be an oriented knot diagram of a knot $K$. Let  $\gamma^\dagger$ be a subarc of $K$ such that 
    \begin{enumerate}
        \item $\gamma^\dagger$ has exactly one self-crossing in $D$, 
        \item when $\gamma^\dagger$ is traced according to the knot orientation, $\gamma^\dagger$ meets only undercrossings or only overcrossings in $D$ and these are the only crossings that $\gamma^\dagger$ meets other than the self-crossing, and
        \item the two one-valent edges incident to the self-crossing lie in the exterior of the disc bounded by the loop formed by $\gamma^\dagger$.
    \end{enumerate}
    Label the self-crossing of $\gamma^\dagger$ with a positive integer $u$. Following the orientation of the knot, label the crossings (if exist) that are distinct from $u$ with consecutive integers $(u+1), \ldots, v$ respectively.

    Let $\gamma$ be a subarc of $\gamma^\dagger$ such that one endpoint of $\gamma$ lies between Crossings $u$ and $u+1$, and the other endpoint of $\gamma$ lies between Crossings $v$ and $u$. A \emph{loop detour move} is a $(v-u,0)$-tangle detour move together with a Type~I Reidemeister move on Crossing $u$ that immediately follows.

    We call the sub-diagram consisting of the loop $\gamma^{\dagger}$ and the $(v-u,0)$-tangle before performing the loop detour move a \emph{loop-tangle diagram}. See \reffig{LoopTangleDiag} for example. 
\end{definition} 

\begin{figure}[ht]
    \centering
    \includegraphics[width=0.6\textwidth]{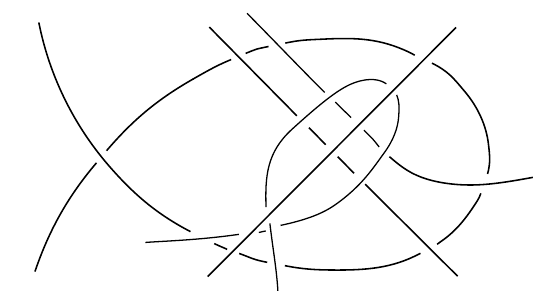}
    \caption{An example of a loop-tangle diagram}
    \label{Fig:LoopTangleDiag}
\end{figure}

\begin{lemma} \label{Lem:LoopCode}
    Performing a loop detour move on an ascending knot diagram $D$ removes $(v-u+1)$ crossing labels $u,u+1,\ldots,v-1,v$ from the Gauss code of $D$. Specifically, this move deletes the subword $u^{-1}(u+1)^{-1}\ldots v^{-1}u^{+1}$ corresponding to traversing the loop from crossing $u$ to itself, and thereby removes the labels $(u+1)^{+1},(u+2)^{+1},\ldots,(v-1)^{+1},v^{+1}$ from wherever they appear in the code. This leaves the order and sign of all other labels unchanged. 
\end{lemma}  

\begin{proof}
    After the loop detour move, the self-crossing of the loop and all other crossings on the loop will disappear. This corresponds to the removal of the $(v-u+1)$ crossing labels $u,u+1,\ldots,v-1,v$, no matter what superscripts they have, in the Gauss code. 
\end{proof}

Typically, after performing such a move, the crossings will be relabelled such that the crossing numbering is once again consecutive from 1 to the number of crossings; i.e. after removing crossings $u,\ldots,v$ from a diagram with $n$ crossings, relabel $v+1,\ldots,n$ with $u,\ldots, n-(v-u+1)$ respectively.

\section{The unknotting robot}\label{Sec:UnknottingKnots}

In this section, we provide the Flatlander's combinatorial proof of \refthm{UnknottingKnots} in two parts: \refprop{PostRobotIsAscending} and \refthm{AscendingIsUnknot}. In particular, \refthm{AscendingIsUnknot} is a new combinatorial proof that ascending knot diagrams are unknotted. We also extend the robot to links.

\begin{definition}
    Given an oriented knot diagram  $D$ of a knot $K$ in the $3$-sphere, consider a \textit{robot} placed at a starting point $b$. The robot then follows the knot consistent with the orientation.
    
    The robot operates on crossings as it meets them in the following way (see \reffig{ActionofRobot}): 
        \begin{enumerate}
            \item When the robot meets an overcrossing, it does nothing to the crossing. 
            \item When the robot meets an undercrossing, it modifies the crossing to be an overcrossing in the direction of travel. 
        \end{enumerate} 

    The robot terminates when it returns to $b$. As a convention, we allow the robot to number the crossings as it meets them for the first time. Note also that the robot may be run \emph{backward} from a given starting point $b$ on a diagram, and follow the knot counter to the orientation; this is equivalent to running the robot on the reverse orientation diagram. 
\end{definition}

\begin{remark}\label{rmk:1,1_tangle_view}
    The starting point $b$ may be moved to infinity in $\SS^2$ by planar isotopy to consider the diagram instead as a $(1,1)$-tangle diagram. This view is used in Section~\ref{Sec:Simplification}. This is purely for convenience and does not affect the result in, for example, $\RR^3$ where that specific planar isotopy is not available.
\end{remark} 

One can also consider a \emph{descending robot}, which is the same in all respects except that it performs the opposite actions at crossings; however, all robots considered in this paper will be as above unless specified.

\begin{definition}
     The diagram obtained by applying the robot to a given oriented knot diagram is the \emph{post-robot diagram}.  
\end{definition}

\begin{lemma}\label{Lem:action_of_robot}
    Consider the post-robot diagram $D$ of an oriented knot diagram with starting point $b$. Then the Gauss code based at $b$ read in the direction the robot was applied is in a form such that every first instance of a crossing $x$ is an undercrossing and the second, an overcrossing.
\end{lemma}

    \begin{proof}
           
        Consider when the robot meets a crossing $x$ for the second instance. If it encounters $x$ as an overcrossing, it does not alter $x$. If it encounters $x$ as an undercrossing in the second instance, it switches the crossing $x$. Note that each crossing is traversed by the robot exactly twice, and each crossing label appears exactly twice in a Gauss code, once as an overcrossing and once as an undercrossing. Hence, in the Gauss code read in the direction the robot was applied, the second instance of every crossing is an overcrossing and the first instance of every crossing is an undercrossing. 
    \end{proof}

\begin{proposition} \label{Prop:PostRobotIsAscending} 
    Applying the robot to an oriented knot diagram $D$ from a starting point $b$ produces a diagram that ascends from $b$ in the forward direction (and consequently descends in the backward direction). 
\end{proposition} 

\begin{proof}
    By \reflem{action_of_robot}, the Gauss code of a post-robot diagram $D$ is organised such that the first appearance of a given crossing is an undercrossing and the second is an overcrossing. By \reflem{GaussCodesForAscending}, the diagram $D$ is ascending from $b$.  

    If the Gauss code is read from $b$ in the reverse direction, crossings are met first as overcrossings then as undercrossings. Again by \reflem{GaussCodesForAscending}, $D$ is descending from $b$.
    \end{proof}

Note again that an ascending diagram depends only on the order crossings are met (determined by the initial diagram), the starting point, and the direction travelled; see \reffig{basepoint-example}. In applying the robot to a given diagram, we can make an appropriate choice of starting point and direction to obtain any ascending diagram with that initial diagram. 

\begin{figure} 
    \centering
    \includegraphics[width=0.7\textwidth]{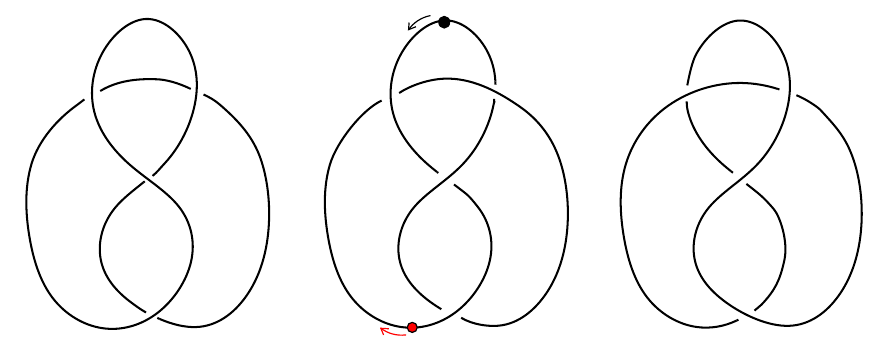}
    \caption{Left figure shows the result of running the robot from the red dot in the direction shown. Right figure shows the result of running the robot from the black dot.}
    \label{Fig:basepoint-example}
\end{figure}

\begin{theorem}\label{Thm:AscendingIsUnknot}
    An ascending or descending knot diagram can be transformed into the zero-crossing unknot diagram by a finite sequence of loop detour moves. 
\end{theorem} 

\begin{proof} 
 
    We will prove the case for ascending knot diagrams. 

    Let $D$ be a knot diagram that ascends from a starting point $b$. Let $n$ be the number of crossings in $D$. Since $D$ is ascending, by \reflem{GaussCodesForAscending}, the first and second appearances of each crossing label are under and over, respectively. By convention, the robot labels each of the $n$ crossings in order as it meets them for the first time. Then there exists a first crossing $u$, $1\leq u \leq n$, traversed twice by the robot after meeting $v$ crossings, $u\leq v \leq n$, in order. The resulting Gauss code is of the form $1^{-1}2^{-1}\ldots u^{-1}\ldots v^{-1}u^{+1}\ldots$; note that this does not exclude the possibility that $u=v$, in which case the loop detour move is simply a standard Reidemeister I move that removes the kink represented by the subcode $u^{-1}u^{+1}$.

    Suppose there are other letter(s) between $u^{-1}$ and $u^{+1}$. The Gauss code is then of the form which allows the application of a loop detour move on the loop starting at crossing $u$. By \reflem{LoopCode}, the resulting code is either $1^{-1}2^{-1}\ldots (u-1)^{-1}(v+1)^{-1}\ldots$ or $1^{-1}2^{-1}\ldots w^{-1}\ldots (u-1)^{-1}w^{+1}\ldots$ for some $w<u$, as the labels $u,\ldots, v$ are removed and the next traversed crossing is either a new crossing traversed for the first time or a previous crossing traversed for the second time. 
    
    After performing the loop detour move, the number of crossings in the new diagram is at least one less than that in the original diagram. We relabel crossings $(v+1),\ldots, n$ with $u,\ldots, n-(v-u+1)$, respectively. This yields a Gauss code that is again composed of consecutive integers from $1$ to the number of crossings in the diagram, arranged such that the first instance of each crossing is an undercrossing. Then there exist labels $u'$ and $v'$ that fill the roles of $u$ and $v$ above; e.g. when the resultant Gauss code is of form $1^{-1}2^{-1}\ldots w^{-1}\ldots (u-1)^{-1}w^{+1}\ldots$, then $u'=w$ and $v'=(u-1)$. Again a loop detour (or Reidemeister I) move may be applied that removes at least one crossing. As the number of crossings $n$ is finite, we can obtain an empty Gauss code after applying finitely many loop detour moves, i.e. all crossings are removed and the resulting diagram is a zero-crossing unknot diagram. Thus, any ascending diagram depicts a knot that is ambient isotopic to the unknot.

    The argument is identically structured for descending diagrams generated by applying a descending robot numbering crossings as it meets them.  The only modifications required are interchanging the words ``undercrossing'' and ``overcrossing'' wherever they appear, swapping the index of each label in each Gauss code to match, and replacing each instance of ``ascending'' with ``descending''.
\end{proof}

\subsection{The unknotting robot unlinks links}\label{Sec:UnlinkingLinks} 

From \refthm{UnknottingKnots}, we learned that the robot unknots knots in any given knot diagram. A natural question is: What about links? Does the robot unlink links? Our short answer is yes, \refcor{UnknottingLinks} in this section provides a more precise statement. Before showing the argument, let us clarify what we meant by an unlink and unlinking a component as follows.   

\begin{definition}
Let $n$ be a positive integer. An $n$-component \emph{unlink} in the $3$-sphere $\SS^3$ is a union of $n$ circles $\{C_i:i\in\ZZ\cap[1,n]\}$ embedded in $\SS^3$ such that for each $j\in\ZZ\cap[1,n]$, $C_j$ is isotopically trivial in the complement of all other link components $\SS^3\setminus\bigcup_{i\in\ZZ\cap[1,n]\setminus j}C_i$. 

We say a link component $C_j$ is \emph{unlinked from all other components} if and only if $C_j$ is homotopically trivial in $\SS^3\setminus\bigcup_{i\in\ZZ\cap[1,n]\setminus j}C_i$. 
\end{definition} 

Note that, as the robot follows a knot diagram, it never makes a ``jump'' at a crossing - it exits on the same strand as it approached the crossing. Thus, in a link diagram, the robot is confined to a single component.
Hence, applying the robot to an oriented link diagram involves a choice of starting point for each component and the sequence of components that the robot is applied to. 

\begin{lemma}\label{Lem:OneComponent}
    Applying the robot to a link component diagram $D_0$ of an oriented link diagram $D$ unknots the link component corresponding to $D_0$ and unlinks it from all other components.
\end{lemma}

    \begin{proof}
        Let $D_1$ be any link component diagram of $D$ other than $D_0$. When the robot reaches a crossing $x$ between the components $D_0$ and $D_1$, it switches $x$ such that $D_0$ crosses over (if $D_0$ already crosses over then the robot does nothing). As the robot is confined to a single component, it never encounters $x$ from $D_1$, and $x$ remains an overcrossing from $D_0$. Hence, $D_0$ will lie completely above all other components after the robot operation and is thus unlinked from all other components. 

        By applying \refthm{UnknottingKnots} to the link component diagram $D_0$, the robot produces 
        a knot diagram that corresponds to an unknot. 
    \end{proof}

\begin{lemma} \label{Lem:Stacking}
    Let $L$ be an $n$-component link and $L_1$, \ldots , $L_n$ be the link components of $L$. Let $D$ be a link diagram of $L$ and $D_1$, \ldots , $D_n$ be the corresponding link component diagrams. Suppose the robot acts on each of the link component diagrams in the order $D_1$, \dots , $D_n$. Then, we obtain a stack of knot diagrams with order (from bottom to top) $L_1$, \ldots , $L_n$.  
\end{lemma} 

\begin{proof}
    Suppose the robot acts on a link component diagram $D_i$.
    When the robot comes across a crossing $c$ with
    a strand coming from a distinct link component diagram $D_j$, it will meet $c$ exactly once along $D_i$, resulting in $D_i$ being positioned above $D_j$ at $c$. 
    
    As a result, 
    the robot pushes the diagram $D_i$ to the forefront. The corresponding link component $L_i$ then becomes unlinked from the rest of the diagram. After applying the robot to each link component diagram in the order $D_1$, \ldots , $D_n$, we obtain a stack of component diagrams in disjoint projection planes, and they have order $L_1$, \ldots, $L_n$ from bottom to top. 
\end{proof} 

\begin{figure}[ht]
$$\vcenter{\hbox{\begin{overpic}[scale = 1]{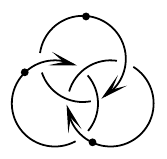}
\put(41, 72){$1$}
\put(39, -4){$2$}
\put(5, 45){$3$}
\end{overpic} }}
  \xrightarrow[]{Robot_1} \
\vcenter{\hbox{\begin{overpic}[scale = 1]{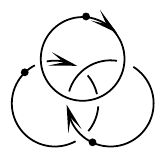}
\put(41, 72){$1$}
\put(39, -4){$2$}
\put(5, 45){$3$}
\end{overpic} }}     \xrightarrow[]{Robot_2} \
\vcenter{\hbox{\begin{overpic}[scale = 1]{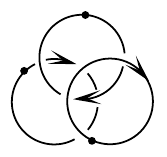}
\put(41, 72){$1$}
\put(39, -4){$2$}
\put(5, 45){$3$}
\end{overpic} }}
  \xrightarrow[]{Robot_3} \
\vcenter{\hbox{\begin{overpic}[scale = 1]{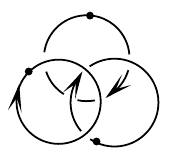}
\put(41, 72){$1$}
\put(39, -4){$2$}
\put(5, 45){$3$}
\end{overpic} }}$$
    \caption{Illustration of the robot's action, from left to right, on the components of a link.}
    \label{fig:robotonlinkcomponents}
\end{figure}

\begin{proposition} \label{Prop:PostRobotIsAscendingLinkCpn} 
    If we apply the robot to each link component of a link in an oriented link diagram, then the robot produces a link diagram with link components lying in disjoint parallel projection planes, each projection plane contains exactly one link component.  The post-robot diagram of each link component diagram $D_i$ ascends from its starting point $b_i$ in the forward direction and descends in the backward direction. 
\end{proposition} 

\begin{proof} 
    Let $D$ be an oriented link diagram of a link $L$. Suppose the robot is applied to a link component diagram $D_i$ of the link component $L_i$ from a starting point $b_i$ in $D$.  

   By \reflem{Stacking}, we separate all link components of $L$ into disjoint parallel projection planes. 

    By \refprop{PostRobotIsAscending}, the post-robot diagram of each $D_i$ ascends from $b_i$ in the forward direction and descends in the backward direction. 
\end{proof} 

\begin{theorem}\label{Thm:AscendingIsUnlink}
    A link diagram that consists of ascending or descending link component diagrams in disjoint parallel projection planes is a diagram of an unlink. 
\end{theorem} 

\begin{proof} 
    Let $D$ be the link diagram of a link $L$ such that $D$ consists of ascending or descending link component diagrams in disjoint parallel projection planes. 
    
    Since the knot diagram of each link component of $L$ in $D$ lies in exactly one projection plane that is disjoint from all other parallel projection planes, we can slide the projections planes to new positions such that all link components are separated when they are projected to a common projection plane. 
    
    Note that sliding the projection planes in such a way constitutes an ambient isotopy that separates the individual link components in the diagram. The diagram $D$ thus becomes a disjoint union of ascending or descending knot diagrams. 
    
    By \refthm{AscendingIsUnknot}, an ascending or descending knot diagram is a diagram of an unknot. Hence, the link $L$ is an unlink. 
\end{proof} 

\begin{corollary}\label{Cor:UnknottingLinks}
    Let $D$ be any link diagram of a link $L$ in the $3$-sphere. Successively applying the robot to each link component produces a link diagram of an unlink. 
\end{corollary} 

\begin{proof}
    The statement follows from \refprop{PostRobotIsAscendingLinkCpn} and \refthm{AscendingIsUnlink}. 
\end{proof}

\section{An upper bound on the number of Reidemeister moves}\label{sec:upper_bound}

In this section, we provide an upper bound on the minimal number of Reidemeister moves required for transforming an ascending or descending knot diagram to the zero-crossing unknot diagram. 

To show the upper bound in \refthm{UpBdRMoves}, we will first prove an upper bound on the minimal number of Type II and Type III Reidemeister moves that a tangle detour move (see \refdef{TangDetourMove}) can be decomposed into. Such an upper bound is stated in \refprop{TangleDetourMove}. Note that the tangle diagram $T$ in \refprop{TangleDetourMove} can be a subdiagram of an arbitrary link diagram, which need not necessarily be an ascending or descending knot diagram. 

\begin{proposition} \label{Prop:TangleDetourMove}
    Let $T$ be an $(m,n)$-tangle diagram. Let $c$ and $t$ be the number of crossings and number of trivial arcs in $T$ respectively. The tangle detour move associated to $T$ can be decomposed into a sequence of at most $(5c+m+n-t)$ Type II and Type III Reidemeister moves. 
\end{proposition}   

\begin{proof}
Let $e$ be the number of non-sprout edges in $T$. Let $\gamma$ be the detour arc of the tangle detour move associated to the $(m,n)$-tangle diagram $T$. Let $\mathcal{D}$ be the closed disc that $\gamma$ sweeps over, we assume $\mathcal{D}$ to be continuously deforming according to the position of $\gamma$ such that $\gamma$ is always a subarc of $\partial\mathcal{D}$. 

Let $G_T$ be the $(m,n)$-tangle graph of $T$. By \refdef{TangDetourMove}, we can assume that the two endpoints of $\gamma$ are fixed in the projection plane of the diagram. 

Consider the graph $G_{T\mathcal{D}} \coloneqq G_T \cup \partial\mathcal{D}$ with $(m+n+2)$ vertices in $\partial \mathcal{D}$. (See \reffig{G_TD_And_RII}, left.) Observe that $\gamma \subset \partial\mathcal{D}$ and $\gamma$ intersects $m$ sprouts of $T$ before the tangle detour move. 

\begin{figure} 
\begingroup%
  \makeatletter%
  \providecommand\color[2][]{%
    \errmessage{(Inkscape) Color is used for the text in Inkscape, but the package 'color.sty' is not loaded}%
    \renewcommand\color[2][]{}%
  }%
  \providecommand\transparent[1]{%
    \errmessage{(Inkscape) Transparency is used (non-zero) for the text in Inkscape, but the package 'transparent.sty' is not loaded}%
    \renewcommand\transparent[1]{}%
  }%
  \providecommand\rotatebox[2]{#2}%
  \newcommand*\fsize{\dimexpr\f@size pt\relax}%
  \newcommand*\lineheight[1]{\fontsize{\fsize}{#1\fsize}\selectfont}%
  \ifx\svgwidth\undefined%
    \setlength{\unitlength}{355.51246486bp}%
    \ifx\svgscale\undefined%
      \relax%
    \else%
      \setlength{\unitlength}{\unitlength * \real{\svgscale}}%
    \fi%
  \else%
    \setlength{\unitlength}{\svgwidth}%
  \fi%
  \global\let\svgwidth\undefined%
  \global\let\svgscale\undefined%
  \makeatother%
  \begin{picture}(1,0.30299737)%
    \lineheight{1}%
    \setlength\tabcolsep{0pt}%
    \put(0,0){\includegraphics[width=\unitlength,page=1]{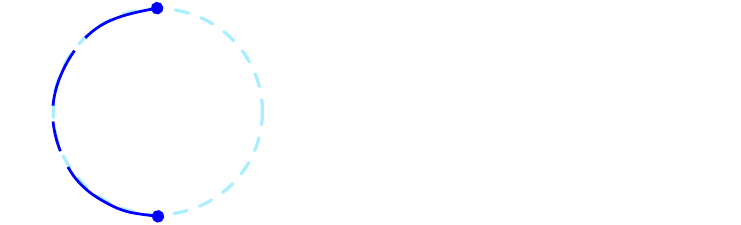}}%
    \put(0.082804,0.02740707){\color[rgb]{0,0,1}\makebox(0,0)[lt]{\lineheight{1.25}\smash{\begin{tabular}[t]{l}$\gamma$\end{tabular}}}}%
    \put(0.27854042,0.28617089){\color[rgb]{0,0.8,1}\makebox(0,0)[lt]{\lineheight{1.25}\smash{\begin{tabular}[t]{l}$\partial\mathcal{D}$\end{tabular}}}}%
    \put(0.18951498,0.15224928){\makebox(0,0)[lt]{\lineheight{1.25}\smash{\begin{tabular}[t]{l}$P$\end{tabular}}}}%
    \put(0,0){\includegraphics[width=\unitlength,page=2]{G_TD_And_RII.pdf}}%
    \put(0.64587698,0.02740702){\color[rgb]{0,0,1}\makebox(0,0)[lt]{\lineheight{1.25}\smash{\begin{tabular}[t]{l}$\gamma$\end{tabular}}}}%
    \put(0,0){\includegraphics[width=\unitlength,page=3]{G_TD_And_RII.pdf}}%
  \end{picture}%
\endgroup%
  
\caption{Left: The graph $G_{T\mathcal{D}}$ consists of the $(3,5)$-tangle graph $G_T$ (black) and the boundary $\partial \mathcal{D}$ of the disc that the detour arc $\gamma$ will sweep over. The symbol $P$ denotes an $8$-sided polygon in $G_{T\mathcal{D}}$. Right: The subdiagram (or subgraph) obtained after applying three Type II Reidemeister moves to the diagram on the left. Note that $\partial\mathcal{D}$ always contains $\gamma$ and its two end points. }
\label{Fig:G_TD_And_RII} 
\end{figure}  

Let $P$ be any polygon in $G_{T\mathcal{D}}$ such that one of the sides of $P$ is a subarc of $\gamma$ and $P$ lies in $\mathcal{D}$. (See \reffig{G_TD_And_RII}, left.) By the definition of detour move, $\gamma$ meets only undercrossings or only overcrossings, hence $\gamma$ cannot cross itself to form a $1$-sided polygon. Thus, $P$ cannot be $1$-sided. 

By definition, a tangle diagram is a sub-diagram of some knot diagram, so a tangle graph is a sub-graph of some $4$-valent graph. Thus, $P$ is a polygon or sub-polygon in some $4$-valent graph that contains $G_T$ as a subgraph. By \reflem{NoCommonEdge}, any two sides of $P$ that comes from $G_T$ are disjoint edges in $G_T$. 

Suppose $P$ is a $2$-sided polygon. There are two cases. If the other side of $P$ is a subarc of $\partial\mathcal{D}$, then $G_T$ is a $(0,0)$-tangle graph. This implies that the associated tangle detour move consists of no Reidemeister moves. Suppose, on the other hand, that the other side $S$ of $P$ is an edge or a subarc of a trivial arc in $T$. We can then apply a Type~II Reidemeister move to remove $S$ from
$\mathcal{D}$. 
 
Suppose $P$ is a $3$-sided polygon (i.e. triangle). From the definition of tangle detour move, the triangle $P$ cannot have two vertices that come from the same crossing, so the third vertex must come from $G_{T\mathcal{D}}$. By the definition of $G_{T\mathcal{D}}$, the other two sides of $P$ cannot be disjoint subarcs of $\partial\mathcal{D}$, so $P$ must have some side(s) that comes from $G_T$. 

If the triangle $P$ has exactly one side $S_T$ that is an edge or a subarc of trivial arc in $G_T$, then $S_T$ is a sprout or trivial arc in $G_T$. We can then perform an ambient isotopy without any Reidemeister move such that the detour arc is deformed into the post-tangle-detour-move position. Suppose the triangle $P$ has exactly two sides $S_T$ and $S_T'$ that are edges or subarcs of trivial arcs in $G_T$. These sides must be edges because trivial arcs do not have any crossings in $G_T$. Since the detour arc $\gamma$ meets only undercrossings or only overcrossings, we can perform a Type III Reidemeister move, which removes exactly one vertex of $G_T$ from $\mathcal{D}$. 

Now suppose $P$ is a $w$-sided polygon with $w>3$. Let $S_{\gamma}$ be the side of $P$ that is a subarc of the detour arc $\gamma$. Except for the two sides of $P$ that are neighbours to $S_{\gamma}$, we apply a Type II Reidemeister move between $S_{\gamma}$ and each of the at most $(w-3)$ $G_T$-edges of $P$. (See \reffig{G_TD_And_RII}, right.)  There would be at most $(w-2)$ triangles created. Since $\gamma$ meets only undercrossings or only overcrossings, we can apply Type III Reidemeister moves to remove at most $(w-2)$ crossings (some of the crossings could be the same crossing). 

Note that $P$ may have distinct vertices sharing the same crossing in $G_T$. (See \reffig{SharedCrossing}, top left.) For every such shared crossing $x$, two of the four edges incident to $x$ will correspond to at most two Type~II Reidemeister moves (see \reffig{SharedCrossing} for example). 

\begin{figure} 
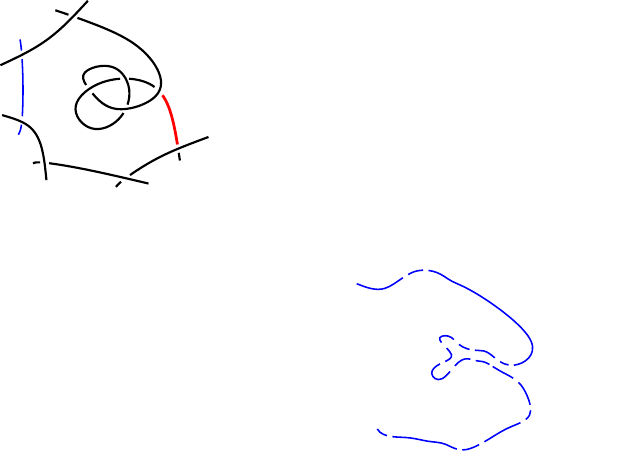  
\caption{A sequence of Reidemeister moves that push the detour arc $\gamma$ away from an $11$-sided polygon $P$.  The crossing $x$ is the shared crossing of $P$. The two red edges incident to $x$ are the two edges that experience exactly two Type~II Reidemeister moves. Each of the other black edges of $P$ experience at most one Type~II Reidemeister move. }
\label{Fig:SharedCrossing} 
\end{figure}  

In any case, either $\gamma$ intersects a sprout/trivial arc, or we can remove edges/crossings/trivial arcs away from the disc $\mathcal{D}$.  Observe that each Type~III Reidemeister move corresponds uniquely to a removed crossing of $G_T$. Moreover, each Type~II Reidemeister move corresponds to some edge/trivial arc in $G_T$, and each edge in $G_T$ corresponds to at most two Type~II Reidemeister moves.

We can repeat a similar process for any polygon in $\mathcal{D}$ with one of its sides a subarc of the detour arc $\gamma$. We repeat the process until the disc $\mathcal{D}$ does not have any non-sprout edges or crossings such that $\gamma$ can be ambient isotoped into $\partial\mathcal{D}\setminus\gamma$ without any Reidemeister moves. Hence, the tangle detour move will require at most \[c+2[e+(m+n-2t)]+t\] 
Type~II and Type~III Reidemeister move(s), where $e+(m+n-2t)$ is the sum of the number of non-sprout edges and the number of sprouts. Note that $[e+(m+n-2t)]$ is the total number of edges in $G_T$. 

By \reflem{RelatingVAndE}, we have $4c = 2e + m+n-2t$. We thus have 
\[2e = 4c-m-n+2t\] 
Hence, an upper bound on the minimal number of Type~II and Type~III Reidemeister moves that a tangle detour move can decompose into is
\begin{align*}
&c+2e+2(m+n-2t)+t \\
&= c+(4c-m-n+2t)+2(m+n-2t)+t\\ 
&= 5c+m+n-t \qedhere
\end{align*} 
\end{proof}

The following is a corollary of \refprop{TangleDetourMove} that shows an upper bound on the minimal number of Reidemeister moves that a loop detour move can be decomposed into. 

\begin{corollary}\label{Cor:decomp_loop_detour}
    Let $T$ be an $(m,0)$-tangle diagram. Let $c$ and $t$ be the number of crossings and number of trivial arcs in $T$ respectively. The loop detour move associated to $T$ can be decomposed into a sequence of at most $(5c+m-t+1)$ Reidemeister moves.
\end{corollary} 

\begin{proof}
    Note that a loop detour move is a tangle detour move together with a Reidemeister~I move that immediately follows. Using this fact and \refprop{TangleDetourMove}, we have that a loop detour move can be decomposed into a sequence of at most $(5c+m-t+1)$ Reidemeister moves.   
\end{proof}

From \refcor{decomp_loop_detour} and the proof of \refthm{AscendingIsUnknot}, we obtain another main result below:  

\begin{theorem}\label{Thm:UpBdRMoves}
        Any ascending or descending knot diagram with $C$ crossings can be transformed into the zero-crossing unknot diagram using at most $(7C+1)C$ Reidemeister moves. 
\end{theorem}

\begin{proof}
Let $D$ be any ascending or descending knot diagram with $C$ crossings and $E$ edges. 
The proof of \refthm{AscendingIsUnknot} shows that $D$ can be transformed into the zero-crossing unknot diagram by at most $C$ loop detour moves. 

By \refcor{decomp_loop_detour}, each loop detour move associated to some $(m,0)$-tangle diagram $T$ can be decomposed into a sequence of at most $(5c+m-t+1)$ Reidemeister moves, where $c$ and $t$ are the number of crossings and number of trivial arcs in $T$ respectively. Observe that
\begin{align*}
    5c+m-t+1 &\leq 5c+m+1  \\
    &\leq 5C+E+1 \\
    &= 5C+2C+1 \\
    &= 7C+1
\end{align*}

Since each loop detour move removes at least one crossing from $D$, there are at most $C$ loop detour moves that transform $D$ into the zero-crossing unknot diagram.  We thus obtain $(7C+1)C$ as an upper bound for the minimal number of Reidemeister moves required. 
\end{proof}
\section{A new alternative proof for the simplification theorem}\label{Sec:Simplification}

\begin{definition}
    A sequence of Reidemeister moves on a diagram is a \emph{simplification} if no new crossings are introduced at any point - at every move, the number of crossings is either reduced or remains constant.
\end{definition} 

We note that the process described in Section \ref{sec:upper_bound} is not a simplification in general. \reffig{nestedfish} also shows a diagram where any detour move on the marked edge is not a simplification, as there are no $3$-gons adjacent to the detour arc. 

\begin{figure}[ht]
    \centering
     $$ \vcenter{\hbox{\begin{overpic}[scale=1]{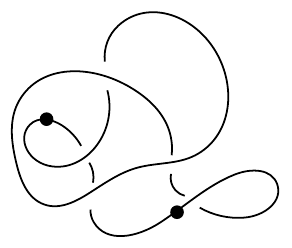}
\end{overpic}}}$$ 
    \caption{Consider the undercrossing strand between the two black dots to be a chosen detour arc for some detour move. This knot diagram provides an example of a detour move that is not a simplification.}
    \label{Fig:nestedfish}
\end{figure}

The loop detour move used to prove that an ascending diagram depicts a knot ambient isotopic to the unknot in \refthm{AscendingIsUnknot} decomposes into a sequence of Reidemeister moves, as shown in \refcor{decomp_loop_detour}. However, even though the result is a diagram with fewer crossings, many extra crossings are introduced in the process. In fact, depending on the loop-tangle diagram, it is not always possible to decompose the loop detour into a sequence of Reidemeister moves that is a simplification.

\begin{lemma}\label{Lem:loop_detour_not_simple}
    A loop detour move is not always a simplification by Reidemeister moves.
\end{lemma}
\begin{figure}
    \centering
    \includegraphics[width=0.6\textwidth]{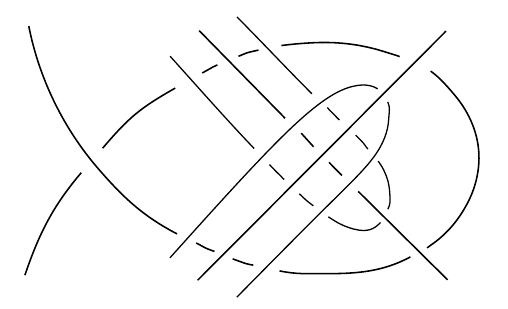}
    \caption{An example of a loop-tangle diagram for which a loop detour move is not a simplification.}
    \label{Fig:nonSimplification}
\end{figure}
    \begin{proof}
        Counterexample: \reffig{nonSimplification}. Moving only the outer loop requires performing a Reidemeister II move that introduces new crossings, as there are no available Reidemeister III moves on the outer loop. 
    \end{proof}

\begin{conjecture}\label{Conj:Simplify}
    It \emph{is} always possible to completely simplify an ascending or descending knot diagram to a no-crossing diagram by Reidemeister moves. 
\end{conjecture} 

It is worth noting that this conjecture is false for links, despite having an analogous robot unlinking theorem. 

\begin{theorem}
    It is not always possible to simplify an ascending link into the unlink by Reidemeister moves.
\end{theorem} 

    \begin{proof}
        Given any ascending knot diagram $D$ with at least one crossing, add a crossing-free link component (collectively, $L$) to the interior of every interstitial region. An example is shown in Figure \ref{fig:simplifying_link_counter}.
        
        Then to perform any Reidemeister move to unknot $D$, it is necessary to pass a strand of $D$ over some component of $L$. New crossings between $D$ and $L$ must be introduced during this process. Hence, there  exists at least one Reidemeister move that increases the number of crossings.  
    \end{proof}

\begin{figure}
    \centering
    \includegraphics{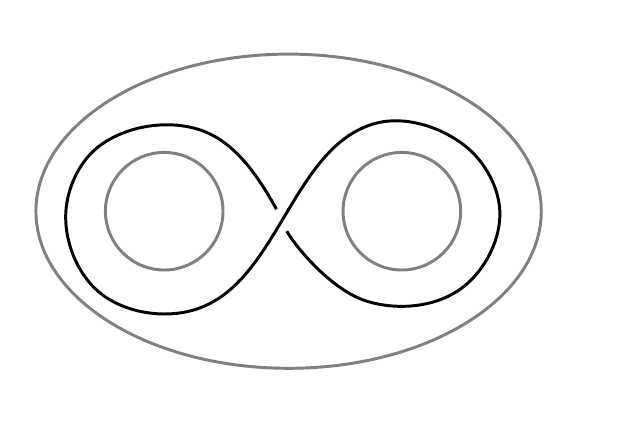}
    \caption{A one-crossing unknot diagram $D$ (black) with added link components $L$ (grey) that cannot be simplified by Reidemeister moves.}
    \label{fig:simplifying_link_counter}
\end{figure}

We return to the knot case. Despite Lemma~\ref{Lem:loop_detour_not_simple}, we would still like to unknot the diagram via loop detours if possible. 
We establish a goal state with the following lemma:  

\begin{lemma}\label{lem:detour_goal}
    A loop detour performed on a loop-tangle diagram with no interior crossings is a simplification by Reidemeister moves.
\end{lemma}

    \begin{proof}
        Let $D$ be a loop-tangle diagram based at a crossing $u$ with no interior crossings, that is, no crossings in the tangle. Each strand in the tangle divides the disc bounded by the loop into two regions. Consider the regions which do not have $u$ in the boundary. There exist innermost regions that are bigons; performing an RII move on such a region removes the strand from the loop-tangle diagram. When all strands have been removed, the loop detour is completed with an RI move on the 1-gon with $u$ as its vertex.
    \end{proof}

\begin{figure}
    \centering
    \includegraphics{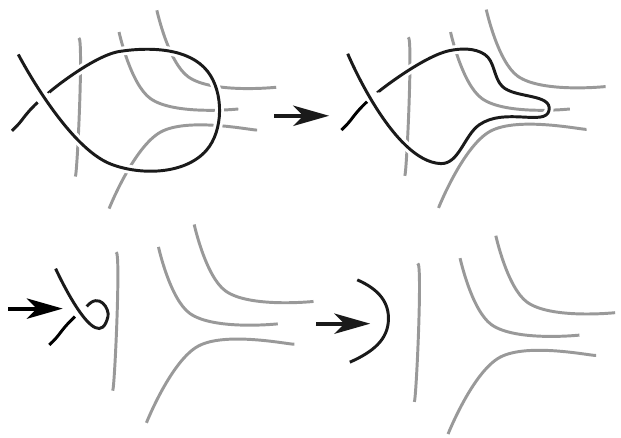}
    \caption{Performing a loop detour on a loop-tangle diagram with no internal crossings as a simplification via a series of Reidemeister II moves followed by a single Reidemeister I.}
    \label{fig:simplifying_no_crossing}
\end{figure}

An example of a simplifying detour move is shown in Figure \ref{fig:simplifying_no_crossing}. Thus, we aim to remove the crossings from the interior of a loop-tangle diagram by simplification.

\begin{lemma}\label{lem:R-moves_in_ascending}
    In an ascending diagram, if a 2- or 3-sided region is not available to perform Reidemeister moves on, it contains the starting point $b$ in its boundary. Further, there are at most two such regions in a given diagram.
\end{lemma}

    \begin{proof}
        Consider the action of a robot as it produces an ascending diagram.
        
        An RII move cannot be performed on a 2-sided region of a diagram only when each edge has one over- and one undercrossing. If $b$ is not on the boundary, then one strand is traversed completely by the robot before the other and both crossings are met as undercrossings, so the RII move is available. Suppose instead $b$ is between the two crossings of the 2-sided region on one of the strands. Then one of the crossings is the first crossing met as the robot is applied, and so an undercrossing, and the other is the last crossing met before the robot returns to $b$ and so an overcrossing - the RII move is not available, as shown in Figure~\ref{fig:woven_reidemeister} (left).

        Similarly, an RIII move is unavailable only when each edge of a 3-sided region of the diagram has one over- and one undercrossing. If $b$ is not included in the boundary, the three strands are traversed in some order by the robot - the first with two undercrossings and the third with two overcrossings, so the RIII move is available. Suppose again $b$ is between two of the crossings. Then one crossing is an undercrossing and the other an overcrossing, as above. Depending on the order of traversal of the other two strands, the RIII may be unavailable, as shown in Figure~\ref{fig:woven_reidemeister} (right).

        As $b$ lies on the knot diagram, it is included in the boundary of exactly two regions of the diagram. As such a region is the only type that can be unavailable to perform Reidemeister moves on, at most two regions may be unavailable to perform Reidemeister moves.
    \end{proof}

    \begin{figure}[ht]
            \begin{center}
            $$ \vcenter{\hbox{\begin{overpic}[scale=1]{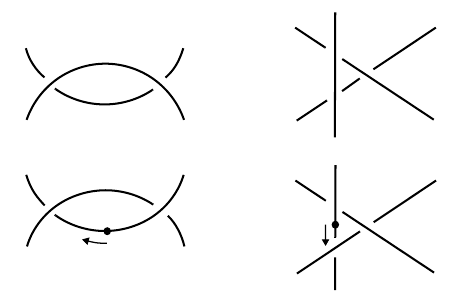}
            \put(49, 39){$b$}
            \put(154, 40){$b$}
            \end{overpic}}}$$
            \end{center}
            \caption{In an ascending diagram, regions that do not contain $b$ are available to perform Reidemeister moves, whereas regions that contain $b$ in their boundary may not be. Left: available and unavailable RII moves. Right: available and unavailable RIII moves.}
            \label{fig:woven_reidemeister}
        \end{figure}

We now use the view of an ascending diagram described in Remark~\ref{rmk:1,1_tangle_view}, where the starting point $b$ is taken to be infinity by planar isotopy.

\begin{lemma}\label{lem:R-moves_always_allowed_in_tangle}
    Let $D$ be a $(1,1)$-tangle diagram of an ascending knot diagram with starting point $b$ at infinity. Then it is possible to apply a Reidemeister I, II, or III move to every bounded 1-, 2-, or 3-gon in $D$, respectively.
\end{lemma}

    \begin{proof}
        Reidemeister I moves are always available on 1-gons, regardless of crossing information.
       
        By Lemma~\ref{lem:R-moves_in_ascending}, the only Reidemeister moves that may be unavailable in an ascending diagram are those on regions directly adjacent to $b$. As $b$ is taken to be the point at infinity, these two regions are the unbounded ones about the $(1,1)$-tangle diagram, i.e.~every bounded  2- and 3-gon of the diagram is available to perform Reidemeister moves upon.
    \end{proof}

Given the above, we define the following:

\begin{definition}
    A \emph{minimal diagram} is a knot diagram in which every possible Reidemeister I or II moves that remove crossings has been applied. 
\end{definition}

By Lemma~\ref{lem:R-moves_always_allowed_in_tangle}, a minimal ascending $(1,1)$-tangle diagram $D$ contains no 1- or 2-sided regions. In particular, no loop-tangle diagram of $D$ contains $b$ in its interior, and so a minimal loop-tangle diagram also contains no 1- or 2-sided regions. As Reidemeister moves are always available in such a loop-tangle diagram, we may instead regard the diagram as purely a graph without crossing information when convenient.

\begin{definition}
A loop-tangle graph is a sub-graph obtained by disregarding the over/under-crossing information in a loop-tangle diagram (see \reffig{LoopTangleDiag} for example) defined in \refdef{LoopDetour}. Note that the loop-tangle graph contains edges that lie outside of the disc bounded by the loop that contains the tangle.  These outward edges are called \emph{sprouts} and each sprout is connected to one vertex (on the loop) only.  
\end{definition} 

Given a loop tangle diagram of an ascending or descending knot diagram, we can ensure that there exists at least three $3$-sided regions by a graph-theoretic statement in \refthm{RelationBetweenNumNgons} if there are no $1$- or $2$-sided regions. 

\refthm{RelationBetweenNumNgons} is analogous to a   more generalised version of Theorem~2 in \cite{Eliahou-Harary-Kauffman:LuneFreeKnotGraphs}, which is quoted in \refthm{RelationBetweenNumNgons_LinkGraph} below. 

\begin{theorem}[Theorem~2 in \cite{Eliahou-Harary-Kauffman:LuneFreeKnotGraphs}]\label{Thm:RelationBetweenNumNgons_LinkGraph}
    Let $G$ be a connected link graph. For any positive integer $N$, let $f_N$ be the number of $N$-sided face(s) in $G$. We have 
    \[3  f_1 + 2f_2 + f_3 = 8 + \sum_{n=1}^{\infty} nf_{n+4}\]
\end{theorem} 

\begin{theorem}\label{Thm:RelationBetweenNumNgons}
    Let $G$ be a loop-tangle graph and $R$ be the disc bounded by the loop such that $R$ contains the tangle graph. For any positive integer $N$, let $f_N$ be the number of $N$-sided face(s) in $R$. We have 
    \[3  f_1 + 2f_2 + f_3 = 3 + \sum_{n=1}^{\infty} nf_{n+4}\]
\end{theorem} 
 
\begin{proof}
Let $v$ be the total number of vertices in both the loop and the tangle, $e$ be the number of edges in $G$ (excluding the sprouts), and $f$ be the number of faces in the disc region $R$. Following from Euler's polyhedron formula, we have $v-e+f=1$. 

Let $s$ be the number of sprouts in $G$. Note that each vertex has four incident edges (including sprouts). Each non-sprout edge is touching exactly two (possibly the same) vertices. By counting the four edges around every vertex (some edges are possibly the same), we count each edge in the loop and tangle twice and each sprout once. Therefore, we have $4v = 2e + s$. 

Let $\beta$ be the number of boundary edges (i.e. edges in the loop, excluding the sprouts). Note that each edge that is neither a sprout nor a boundary edge is touching exactly two faces. By counting all edges of each face in the disc region $R$, we count each interior edge twice and each boundary edge once. Hence, $\sum_{n=1}^{\infty} n f_n = 2e - \beta$. 

We can now express $v-e+f = 1$ in terms of $f$, $f_n$'s, $s$, and $\beta$. 
\begin{align*}
    v-e+f &= 1 \\
    4v-4e+4f &= 4 \\
    (2e+s)-4e+4f &= 4 \\ 
    4f - 2e + s &= 4 \\
    4f - (\sum_{n=1}^\infty nf_n + \beta) + s &= 4 \\
    4\sum_{n=1}^\infty f_n - \sum_{n=1}^\infty nf_n  + (s-\beta) &= 4.
\end{align*}
Since each tangle arc is traverse to the loop, there is exactly one sprout corresponding to each vertex that comes from the crossing between the loop and the tangle. For the vertex that comes from the self-crossing of the loop, there are exactly two sprouts. Note that the number of vertices and the number of edges along a circular loop are the same, we have $s-\beta = +1$. 

Hence, we have 
\begin{align*}
    4\sum_{n=1}^\infty f_n - \sum_{n=1}^\infty nf_n  + 1 &= 4 \\
    3  f_1 + 2f_2 + 1f_3 + 0f_4 + \sum_{n=5}^\infty (4-n)f_n +1 &= 4 \\
    3  f_1 + 2f_2 + 1f_3 + 0f_4 - \sum_{n-4=1}^\infty (n-4)f_{n-4+4} &= 3 \\
    3  f_1 + 2f_2 + 1f_3 &= 3 + \sum_{m=1}^\infty mf_{m+4}. \qedhere
\end{align*}
\end{proof}

Of note is the implication in the relation of \refthm{RelationBetweenNumNgons} that if there are no 1- or 2-sided faces in $R$, then there must exist at least three 3-sided faces in $R$. In particular, 1- and 2-sided regions may be removed by applying Reidemeister I and II moves, respectively, and the remaining at least three 3-sided regions are available to perform Reidemeister III moves. Also significant is the complete lack of restrictions on the number of four sided regions with respect to the number of other polygons.

In fact, the above relation has further implications for the structure of general loop-tangle diagrams, as shown below.

\begin{proposition}\label{prop:not_fixed}
    Given any minimal loop-tangle diagram, at least one of the following types of subdiagrams must exist: 
    \begin{itemize}
        \item triangles with one edge a connected subarc of the loop,  
        \item triangles sharing an edge with another triangle,
        \item quadrilaterals sharing at least one edge with triangles, or
        \item pentagons sharing at least four edges with triangles.
    \end{itemize}
\end{proposition}

    \begin{proof}       
        Proceed by a \emph{discharging} argument. Consider a loop-tangle diagram with each $i$-gon face labelled by $(4-i)$, i.e.~labelled by the \emph{weight}, or \emph{charge}, which is the multiplicative constant of $f_i$ after subtracting $\sum_{n=1}^{\infty} nf_{n+4}$ from both sides of the equation in Theorem~\ref{Thm:RelationBetweenNumNgons}. By the same theorem, the sum of all these weights  (the \emph{net charge}) of the loop-tangle diagram is 3.

        In a minimal loop-tangle diagram, the only faces with positive charges are triangles, with weight $+1$. Consider a \emph{discharging move} where a triangle $T$ distributes $1/3$ of its charge to each of the interior faces of the loop-tangle diagram adjacent to $T$ through an edge. Note that no charge is passed over the loop, as only interior faces receive charge. This plainly preserves the net charge of the loop-tangle diagram.

        Suppose every triangle is simultaneously discharged in this way. The subdiagrams which retain a positive charge are exactly those in the statement of the proposition, each of which retains a charge of at least $+1/3$ after discharging. It can also be seen that for an $n$-gon with $n\geq 6$, the net charge of the $n$-gon is non-positive even if every edge is adjacent to a triangle; thus, these previous cases are the only ones where the net charge is strictly positive. As the net charge across the entire loop-tangle is +3, any loop-tangle diagram must have at least one of the four types of subdiagrams. 
    \end{proof}

We wish to prove that in a minimal loop-tangle diagram, it is always possible to perform Reidemeister moves that allow progressing the loop detour move as a simplification. Of the four types of subdiagrams in Proposition~\ref{prop:not_fixed}, there are two types which allow us to perform non-crossing-creating Reidemeister moves. The first such type is a triangle adjacent to the loop, at which we can perform Reidemeister III move to remove a crossing from the tangle. The second such type is a triangle adjacent to another triangle, where we can apply Reidemeister III move followed by a Reidemeister II move that removes two crossings.    

Note that when we have a quadrilateral with triangles on opposite sides, we can apply two Reidemeister III moves followed by a Reidemeister II move that removes two crossings. The next remark extends the discussion here.

\begin{definition}
A \emph{track} is a path of quadrilaterals connected by opposite edges, with \emph{length} given by the number of quadrilaterals in the path.
For ease of subsequent arguments, a track may have length zero, i.e.~include only a single edge.
\end{definition}

\begin{remark}\label{rem:track_transport}
    Given a triangle at one end of a track, the triangle can be transported along the track by a sequence of Reidemeister III moves.

    Such a sequence terminates in one of the following three ways, when the track meets:
    \begin{itemize}
        \item a second triangle, forming a bigon that can be removed by an RII move,
        \item the exterior loop, moving a crossing outside the loop-tangle, or
        \item an $n$-gon with $n\geq5$.
    \end{itemize}
\end{remark}

    \begin{figure}[ht]
            \begin{center}
            \includegraphics[]{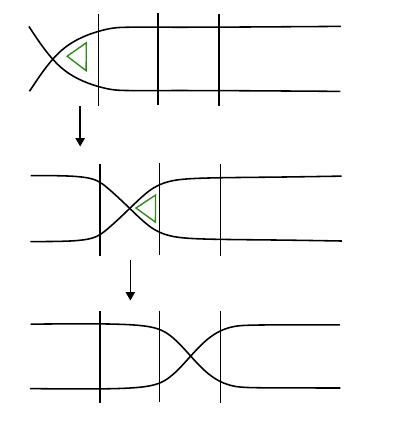}
            \end{center}
            \caption{Given a triangle next to a path of 4-gons, it is possible to transport a crossing along the ``track" by repeated RIII moves. Here, each successive diagram is obtained by performing an RIII move on the triangular region indicated by a green triangle.}  
            \label{fig:crossing_transport}
        \end{figure}

The remaining cases, which do not obviously allow simplification, are a quadrilateral adjacent to a single triangle or to exactly two triangles that share a vertex, or a pentagon adjacent to at least four triangles.

\begin{theorem}\label{thm:always_simplify}
    It is always possible to simplify the interior of a loop-tangle in an ascending diagram such that the associated loop detour may be performed as a simplification.
\end{theorem}

    \begin{proof}
        Let $D$ be a loop-tangle diagram.
        Assume WLOG that $D$ is minimal, as removing 1- and 2-gons by Reidemeister moves is a simplification.

        By Proposition~\ref{prop:not_fixed}, there exists in $D$ at least one of the four types of subdiagrams listed in the same statement.

        Of these, simplification is possible in a fixed number of Reidemeister moves for:
        \begin{itemize}
            \item a triangle adjacent to the loop, 
            \item two adjacent triangles, and
            \item any quadrilateral adjacent to two triangles on opposite sides.
        \end{itemize}

        Suppose none of the above cases occurs. Then, again by Proposition~\ref{prop:not_fixed}, there exist quadrilaterals next to either a single triangle or two triangles that share a vertex  of the quadrilateral, or pentagons adjacent to at least four triangles. Performing an RIII move on the latter forms a quadrilateral next to triangles, as shown in Figure~\ref{fig:remove_pentagon}, so WLOG there exist triangles adjacent to quadrilaterals.

        Now as in Remark~\ref{rem:track_transport}, triangles next to quadrilaterals may be transported by successive RIII moves. Of the three possible terminals (loop, triangle, or $n$-gon with $n\geq 5$), two allow simplification. In the following, we will consider the remaining case where the track terminates in an $n$-gon with $n\geq 5$. 

        Suppose then that every possible track from every triangle terminates in an $n$-gon with $n\geq 5$. By \reflem{NoCommonEdge}, there is no shared edge in each polygon in the diagram, so each $n$-gon has $n$ distinct sides. Since any two distinct tracks cannot join to end at the same side of some polygon, it follows that each side of each triangle corresponds to at most one side of some $n$-gon uniquely. Hence, we have the following inequality:
        
        $$3f_3\leq\sum_{n=5}^{n_{max}}nf_n;\quad(*)$$
        i.e.~the total number of sides in $n$-gons with $n\geq5$ exceeds the total number of sides of triangles. If $(*)$ does not hold, at least one track from a triangle must terminate in one of the two simplifying cases and simplification is possible.

        Denote by $n_{max}$ the number of sides of the polygon with the maximum number of sides in the loop tangle diagram. Let $n_{max}=5$. Then by Theorem~\ref{Thm:RelationBetweenNumNgons}, we have $f_3= 3 + f_5$. If $f_5<5$, Equation $(*)$ does not hold because 
        \begin{align*}
            f_5<5 &\Rightarrow f_5 < \frac{9}{2} \\ &\Rightarrow 2f_5<9 \\
            &\Rightarrow 5f_5<9 + 3f_5 = 3(3+f_5) = 3f_3
        \end{align*}
        Simplification is thus possible by Remark~\ref{rem:track_transport} along tracks that do not terminate at pentagons. 
        
        Now assume $f_5\geq5$. Since $3f_3> 3(f_3-3) = 3f_5$, by the pigeonhole principle, there exist pentagons that are the terminal of tracks from four or more triangles. 

        Let $n_{max}>5$. Observe that $3(n-4)\geq n$ for $n\geq6$; that is, $n$-gons with $n\geq6$ imply  the existence of more triangle sides than they can be the terminals of tracks for, by Theorem~\ref{Thm:RelationBetweenNumNgons}.  The above pigeonhole argument thus still holds, and there exist pentagons that are the terminals of at least four tracks.
        
        Select a pentagon that is the terminal of at least four tracks and perform RIII moves such that the triangles are adjacent to the pentagon. As there are at least four triangles, there are triangles which are adjacent to two others through vertices around the pentagon. Performing an RIII move on such a triangle replaces the pentagon and both adjacent triangles with quadrilaterals. 
        
        Consider the polygon on the third vertex of this triangle before the RIII move. If the polygon on the third vertex is not a quadrilateral, then it does not produce a pentagon after the move and the number of pentagons is reduced by at least one. As there are a finite number of pentagons, eventually $f_5<5$ and $(*)$ does not hold, so simplification is possible. If the polygon is a quadrilateral, it becomes a pentagon after the move and the number of pentagons remains constant. Suppose this occurred for every possible triangle. Then the above pigeonhole argument still holds, and some of these pentagons are the terminals of at least four tracks from triangles; as in the right of Figure~\ref{fig:constant_pentagon}, this implies simplification was already possible through a track with triangles at both terminals.
        
        As there exist simplifying Reidemeister moves in every possible case, eventually all crossings in the interior of the loop-tangle may be removed. Then by Lemma~\ref{lem:detour_goal}, the loop detour is a simplification. Thus it is always possible to simplify the interior of a minimal loop-tangle diagram such that the loop detour is a simplification.
    \end{proof}

\begin{figure}[ht]
    \begin{center}
        \includegraphics[]{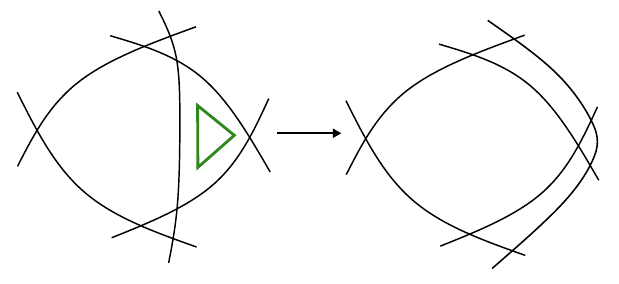}
    \end{center}
    \caption{A pentagon adjacent to three triangles that share vertices can be replaced with quadrilaterals by performing an RIII move on the ``middle" triangle. This region is indicated by a green triangle in the figure.}
    \label{fig:remove_pentagon}
\end{figure}

\begin{figure}[ht]
    \begin{center}
        \includegraphics[]{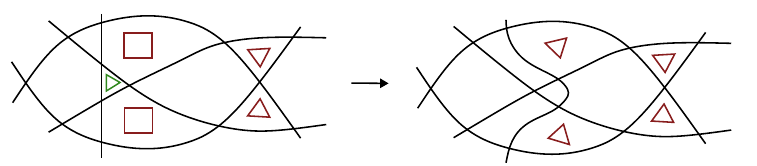}
    \end{center}
    \caption{An arrangement where a Reidemeister III removes one pentagon and forms another. A choice of any three red regions ensures the resulting pentagon is adjacent to at least four triangles, but also that simplification is possible along a track of two quadrilaterals.}
    \label{fig:constant_pentagon}
\end{figure}

Note that the scenario in the above proof of every one of a small number of pentagons being used as efficiently as possible, with every side the terminal of a track from a triangle, is difficult to construct. To construct such an example in practice, a large number of pentagons is usually required, as in Figure~\ref{fig:complicated_loop-tangle}.

\begin{figure}[H]
    \begin{center}
        \includegraphics[width=\textwidth]{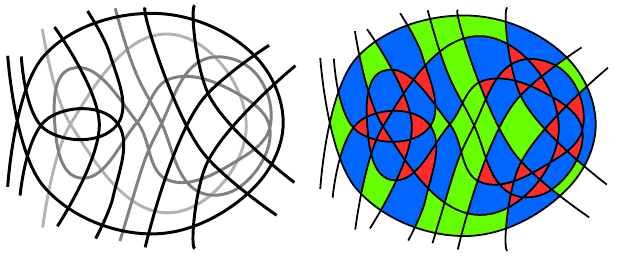}
    \end{center}
    \caption{A loop-tangle diagram where every track from every triangle meets a pentagon. The most complicated components are marked in grey at the left and triangles, quadrilaterals and pentagons are marked in red, green and blue, respectively at right. In this case, $f_5=26$.}
    \label{fig:complicated_loop-tangle}
\end{figure}

\begin{corollary} \label{Cor:Simplified}
    An ascending diagram can always be simplified by Reidemeister moves to a crossingless unknot diagram. 
\end{corollary} 

    \begin{proof}
        By Theorem~\ref{Thm:AscendingIsUnknot}, any ascending diagram can be unknotted by a finite sequence of loop detours. By Theorem~\ref{thm:always_simplify}, it is always possible to simplify the interior of a loop-tangle such that the associated loop detour may be performed as a simplification by Reidemeister moves. Thus, any ascending diagram may be simplified to a crossingless unknot diagram by Reidemeister moves.
    \end{proof}

\begin{remark}
    The above proof suggests a method of simplifying an ascending diagram using a hierarchy of complexity, based on the number of sides of the polygons considered:
    \begin{enumerate}
        \item 1-gons, removed by RI;
        \item Bigons, removed by RII;
        \item Triangles adjacent to either:
            \begin{itemize}
                \item[a.] another triangle - performing RIII forms a bigon;
                \item[b.] the exterior loop - performing RIII removes a crossing from the loop-tangle
            \end{itemize}
        \item Triangles which are $n$ quadrilaterals away from either another triangle or the loop for increasing $n$, transported by successive RIII moves;
        \item Pentagons which are the terminals of tracks from triangles on three adjacent sides, converted to quadrilaterals by an RIII move on a triangle adjacent to two others through vertices.
    \end{enumerate}

    Proposition~\ref{prop:not_fixed} ensures that (5) is the most complex structure that one needs to consider, as in the absence of the previous levels of complexity, such a structure must appear.
    
    To simplify an ascending diagram, move down the above list within the loop-tangle based at crossing $u$ for minimum $u$. When a structure in the list is found, remove it as described and start at the beginning of the list again, until the loop detour is completed. Identify the next loop-tangle and repeat until simplification is achieved and a crossing-free unknot diagram is obtained.
\end{remark}

\section{Interdisciplinary connections} 

In this section, we discuss the applications of previous results or the robot action on electrical networks and DNA topology.

\subsection{Electrical networks}
In \cite{Goldman-Kauffman:KnotsTanglesAElectricalNetworks}, Goldman and Kauffman used a translation of Reidemeister moves on link diagrams to graphical moves on signed graphs as shown in \reffig{ElecNetwork} (see  also \cite[Figure~2.6]{Goldman-Kauffman:KnotsTanglesAElectricalNetworks} or \cite[Figure~11]{Kauffman:CombinatoricsATopology}). The original translation is due to Yajima and Kinoshita \cite{Yajima-Kinoshita_OnTheGraphsOfKnots}. Section~3 of their paper reviewed the background on classical electrical networks which can be modelled by a graph with edges assigned with resistance. The study of such electrical networks motivates the connection between graphical moves and conductance-preserving transformations on networks. More background on the correspondence between link diagrams (with a canonical choice of shaded regions) and signed graphs can be found in \cite[Section~2]{Goldman-Kauffman:KnotsTanglesAElectricalNetworks},  \cite[Section~2.4]{Adams:TheKnotBook}, and \cite[p.355]{Kauffman:KnotsAndPhysics}. The associated unsigned graph for the link diagram is sometimes called a \emph{medial graph} of the link diagram. 

Using the knot-theoretic machinery we have in this paper, we can prove the following graph-theoretic statement. 

\begin{figure}[ht]
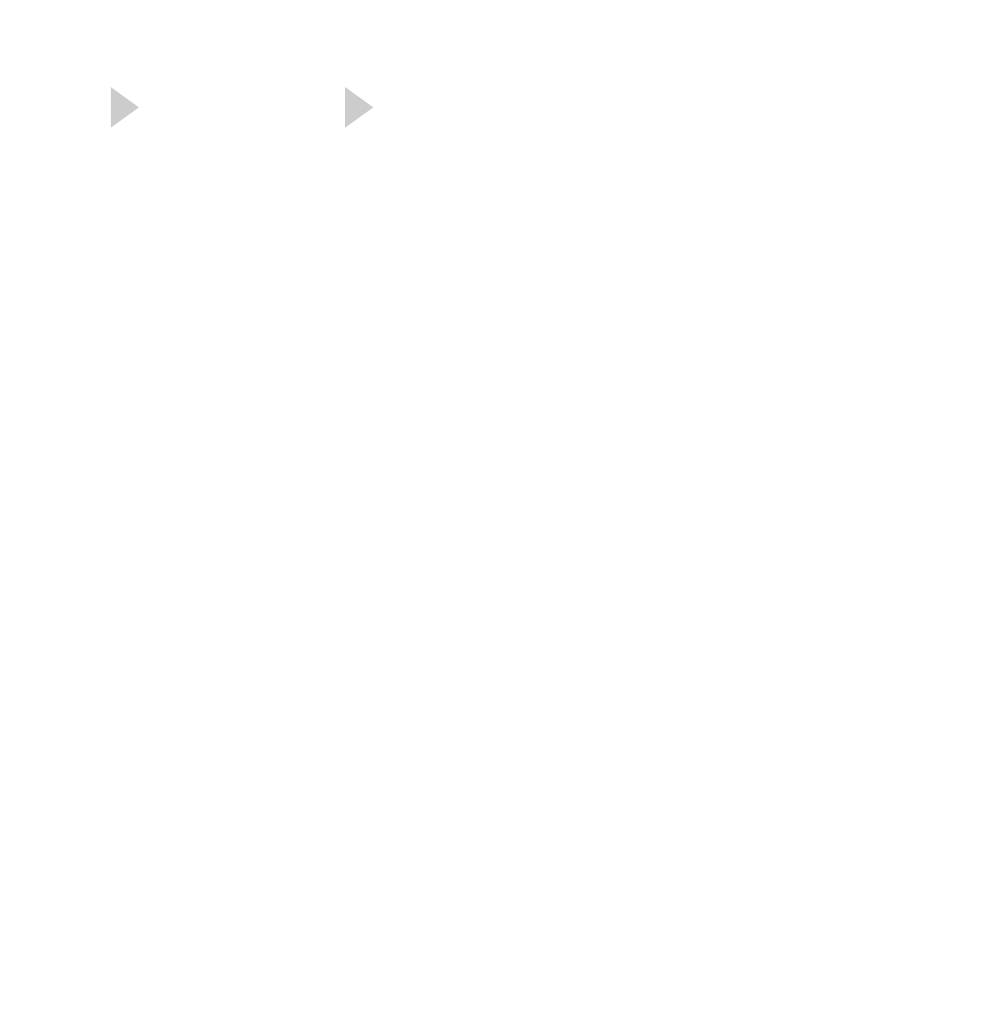  
\caption{Graphical Reidemeister moves. Left:~Signed network moves (or graphical moves). Right:~The corresponding Reidemeister moves with checkerboard colourings.}
\label{Fig:ElecNetwork} 
\end{figure}  

\begin{corollary}\label{Cor:PlanarGraph}
    Any planar graph can be transformed into a finite collection of points via a sequence of the graphical moves shown in \reffig{ElecNetwork} (see  also \cite[Figure~2.6]{Goldman-Kauffman:KnotsTanglesAElectricalNetworks} or \cite[Figure~11]{Kauffman:CombinatoricsATopology}). 
\end{corollary} 

\begin{proof} 
    Let $G$ be a planar graph. By 
    \begin{enumerate}
        \item assigning appropriate signs to each edge of $G$, 
        \item regarding each edge of $G$ as a crossing, and 
        \item regarding each vertex of $G$ as a face in a flat link diagram,
    \end{enumerate}
    we can obtain a link diagram $D$ such that each link component diagram lies completely on top of or beneath each other and is ascending or descending. By \refthm{AscendingIsUnlink}, such a link diagram is a diagram of an unlink. Hence, there is a sequence $\sigma$ of Reidemeister moves that transforms $D$ to a zero-crossing unlink diagram $D_u$, whose associated planar graph is a finite collection of points. Since each Reidemeister move on a link diagram can be translated to a signed move on the associated planar graph, the sequence $\sigma$ that transforms $D$ into $D_u$ can be translated to a sequence $\sigma_g$ of graphical moves that transforms $G$ into a finite collection of points.   
\end{proof}

Note that an explicit sequence of graphical moves that corresponds to a tangle detour move in the proof of \refthm{AscendingIsUnknot} can be obtained using the sequence of Reidemeister moves constructed in \refprop{TangleDetourMove}. 

In translating link diagrams into graphs, crossing types become plus or minus signs on the edges of the graph, corresponding to electrical conductance in the network. 
 Reidemeister II moves become special series and parallel replacements in the graph. That is, these replacements preserve conductivity of the network.  Reidemeister III moves become a star-triangle replacement that also preserves conductivity. (See \reffig{ElecNetwork}, \cite[Figure~2.6]{Goldman-Kauffman:KnotsTanglesAElectricalNetworks}, or \cite[Figure~11]{Kauffman:CombinatoricsATopology}.) Thus, our topological results can be translated to results about planar electrical networks. Goldman and Kauffman \cite{Goldman-Kauffman:KnotsTanglesAElectricalNetworks} use the opposite direction, obtaining invariants of knots and links from electrical networks, while Jaeger's work discussed in \cite{Kauffman:CombinatoricsATopology} studies other invariants from the checkerboard graphs of knots and links.

\subsection{DNA topology}\label{Sec:DNA}
DNA topology is the study of topological properties of DNA (circle DNA) such as winding, knotting, linking, and tangling. These properties are important to study because they affect the process of DNA opening or moving and replicating in the cell which are vital to certain nucleic acid processes. In particular, it has been shown in \cite{DMSZ} that knotted DNA are damaging to cells. We refer the reader to \cite{Sum, DOO, Wang} for more information on DNA topology and the importance of unknotting enzymes such as topoisomerases. 

A simplified expression of DNA can be given as a long double-stranded curve. Enzymes act on the DNA molecule to cut and then rejoin strands — effectively switching crossings. The enzyme action is, however, not orchestrated as with our robot. Consequently, in the DNA, unknotting occurs through forces and properties of the environment of the DNA. Nevertheless, it is interesting to compare DNA unknotting with our robot action on framed knots.

Framed knots and links have special projections, called \emph{blackboard framed projections}, that encode the framing information. 

\begin{definition}
Let $\overrightarrow{D}$ denote a diagram $D$ of a knot with a given orientation. A diagram equipped with \textbf{blackboard framing} is a projection of a framed knot $K$ in $\mathbb{R}^3$ onto $\mathbb{R}^2$ where the normal vector for every point on $K$ is perpendicular to the plane. If $D$ is equipped with blackboard framing, then the framing number of $K$ is the sum of the signs of the crossings of $\overrightarrow{D}$, this is also known as the \textbf{writhe} $w(D)$ of $D$. The \textbf{sign} of a crossing is provided in Figure \ref{fig:sign}.
\end{definition}

\begin{figure}[ht]
    \centering
    \begin{subfigure}{.3\textwidth}
    \centering
$\vcenter{\hbox{\begin{overpic}[scale = 1]{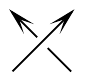}
\end{overpic} }}$
        \caption{$+1$}
\label{fig:possign}
    \end{subfigure}
     \begin{subfigure}{.3\textwidth}
    \centering
$\vcenter{\hbox{\begin{overpic}[scale = 1]{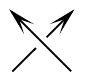}
\end{overpic} }}$
        \caption{$-1$}
\label{fig:negsign}
    \end{subfigure}
    \caption{The sign of a crossing.}
    \label{fig:sign}
\end{figure}

In the blackboard framing setting, isotopy of the framed knot diagram is given by Reidemeister II and III (i.e.~\emph{regular isotopy}) with modified Reidemeister I, as in Figure \ref{fig:modifiedR1}. Let $K$ be a framed knot and $D$ its diagram equipped with blackboard framing. Since the robot operates on knot diagrams we may extend the robot to framed knots and links by applying the robot's operation to blackboard framed projections. By an immediate consequence to \refthm{UnknottingKnots}, the robot unknots framed knots. 

\begin{figure}[ht]
    \centering
    $\vcenter{\hbox{\begin{overpic}[scale = 1]{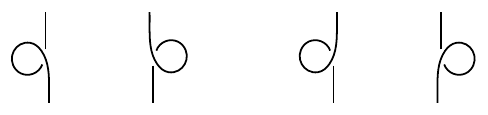}
    \put(35, 25){$\xleftrightarrow{mR1}$}
      \put(175, 25){$\xleftrightarrow{mR1}$}
\end{overpic} }}$
    \caption{The modified Reidemeister I moves.}
    \label{fig:modifiedR1}
\end{figure}

\begin{corollary}
    Let $K$ be any framed knot and $D$ be its diagram equipped with blackboard framing. For any choice of orientation and any starting point of $D$, the post-robot diagram is a diagram of a framed unknot. That is, the robot unknots framed knots.
\end{corollary}

The framing information is not preserved by the action of the robot. In fact, the change in framing is dependent on the projection, starting point, and orientation. See for example in Figure \ref{fig:fig8example2}, the robot is applied to the figure-eight knot with framing number $0$ with the same choice in orientation but different starting points. The first choice yields an unknot with framing number $0$ while the second results in a unknot with framing number $-2$. 

\begin{figure}[ht]
    \centering
    \begin{subfigure}{.48\textwidth}
    \centering
$\vcenter{\hbox{\begin{overpic}[scale = 1]{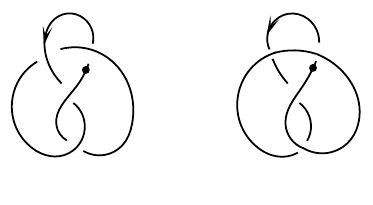}
\put(75, 50){$\xrightarrow[]{Robot}$}
\put(30, 10){$D$}
\put(136, 10){$D_R$}
\end{overpic} }}$
        \caption{$w(D) = 0$ and $w(D_R) = 0$. }
\label{fig:fig8writhe0}
    \end{subfigure}
     \begin{subfigure}{.48\textwidth}
    \centering
$\vcenter{\hbox{\begin{overpic}[scale = 1]{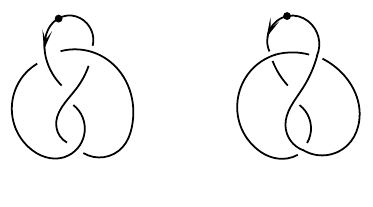}
\put(75, 50){$\xrightarrow[]{Robot}$}
\put(30, 10){$D$}
\put(136, 10){$D_R$}
\end{overpic} }}$
        \caption{$w(D) = 0$ and $w(D_R)=-2$.}
\label{fig:fig8writhen2}
    \end{subfigure}
    \caption{Example of the change in framing between the knot diagram $D$ of $K$ and the post-robot diagram $D_R$.}
    \label{fig:fig8example2}
\end{figure}

Observe that the robots action on a diagram consists of only crossing changes and for each crossing change the writhe changes by an even number. This observation is summarized in the following lemma. 

\begin{lemma}
    The robot changes the framing of a knot by an even number.
\end{lemma}
\begin{proof}
    This is an immediate consequence of the action of the robot, since the robot does not change the number of crossings; the robot only changes crossings and a change in a crossing changes the sign of the crossing from $\epsilon$ to $\pm \epsilon$, where $\epsilon=1$ or $-1$.
\end{proof} 

Buck and Zechiedrich in \cite{BZ} hypothesised an unknotting model of type II topoisomerases on DNA: Type II topoisomerases unknot DNA by removing local hooked juxtapositions formed from the DNA. This hypothesis is backed by simulations \cite{BWFSLocalSR, LMZC, LZCjuxtaposition}. Since DNA is a 3-dimensional object, local hooked juxtapositions are defined in \cite{BZ} by using the curvature and position of the DNA strands. For simplicity we will define a hooked juxtaposition on a projection of a knot to be a bigon in the knot diagram formed from two crossings with the same crossing type. Illustrations of hooked juxtapositions are given in Figure \ref{fig:hookedjux}. \\

We observe a similar behavior from the robot, by comparing the knot diagram $D$ with the post-robot diagram $D_R$, when the starting point is chosen outside of the bigon. More precisely, given a fixed orientation on the knot diagram, the hooked juxtaposition is a bigon consisting of two crossings with the same sign, the bigon of the post-robot diagram contains two crossings consisting of different signs. This observation is illustrated in Figure \ref{fig:hookedjux} and summarised in the following lemma. 

\begin{figure}[ht]
$$\vcenter{\hbox{\begin{overpic}[scale = 1]{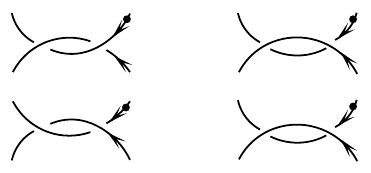}
\put(75, 60){$\xrightarrow[]{Robot}$}
\put(75, 20){$\xrightarrow[]{Robot}$}
\end{overpic} }}
     \qquad \qquad
\vcenter{\hbox{\begin{overpic}[scale = 1]{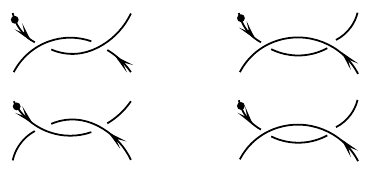}
\put(75, 60){$\xrightarrow[]{Robot}$}
\put(75, 20){$\xrightarrow[]{Robot}$}
\end{overpic} }}$$
    \caption{Illustration of the robot's action, from $D$ to $D_R$, on hooked juxtapositions when the starting point is chosen outside of the two crossings of the bigon.}
    \label{fig:hookedjux}
\end{figure}

\begin{lemma}
    The robot removes hooked juxtapositions if the starting point does not lie between the two crossings of a bigon.
\end{lemma}

The robot is dependent on the chosen starting point. Observe in Figure \ref{fig:hookedjux2}, if the starting point is chosen between the crossings of the bigon then the robot does not remove the hooked juxtaposition.

\begin{figure}[ht]
$$\vcenter{\hbox{\begin{overpic}[scale = 1]{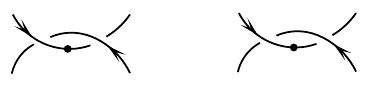}
\put(75, 15){$\xrightarrow[]{Robot}$}
\end{overpic} }}
   \qquad   \qquad
\vcenter{\hbox{\begin{overpic}[scale = 1]{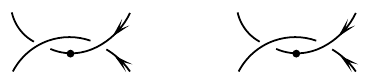}
\put(75, 15){$\xrightarrow[]{Robot}$}
\end{overpic} }}$$
    \caption{Illustration of the robot's action, from $D$ to $D_R$, on hooked juxtapositions when the starting point is chosen between the two crossings of the bigon.}
    \label{fig:hookedjux2}
\end{figure}

Furthermore, observe that a bigon that is not a hooked juxtaposition can be changed to a hooked juxtaposition if the starting point is chosen between the crossings of the bigon, illustrated in Figure \ref{fig:hookedjux3}.

\begin{figure}[ht]
$$\vcenter{\hbox{\begin{overpic}[scale = 1]{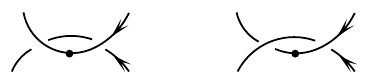}
\put(75, 15){$\xrightarrow[]{Robot}$}
\end{overpic} }}
   \qquad   \qquad
\vcenter{\hbox{\begin{overpic}[scale = 1]{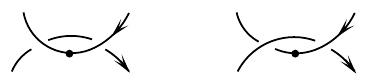}
\put(75, 15){$\xrightarrow[]{Robot}$}
\end{overpic} }}$$
    \caption{Illustration of the robot's action, from $D$ to $D_R$, creating a hooked juxtaposition on a bigon when the starting point is chosen between the two crossings of the bigon.}
    \label{fig:hookedjux3}
\end{figure}
  
However, the robot's extension to links always removes hooked juxtapositions between two different components. This is a consequence of Lemma~\ref{Lem:Stacking}.

\begin{corollary}
    The robot removes hooked juxtapositions between two different link components in a link diagram.
\end{corollary}

\bibliographystyle{amsplain}  

\bibliography{RobotUnknotting_ref.bib} 

\end{document}